\documentclass[a4paper,11pt]{amsart}
\usepackage{comment}
\usepackage{dsfont}
\usepackage{amsmath,amsthm,amssymb,tabu,stmaryrd}
\usepackage{mathtools}
\usepackage[all]{xy}
\usepackage[usenames, dvipsnames]{color}
\usepackage{listings} 
\usepackage{graphicx}
\usepackage[table]{xcolor}
\usepackage{caption}
\usepackage{hyperref}
\hypersetup{
	colorlinks=true,
	linkcolor=OliveGreen,
	filecolor=OliveGreen,
	urlcolor=OliveGreen,
	citecolor=OliveGreen
}
\usepackage[noadjust]{cite}
\usepackage{tikz,tikz-cd}
\usepackage{enumitem}
\usetikzlibrary{shapes.geometric,calc,arrows,patterns}
\usepackage{multirow}
\usepackage{skak}
\usepackage[normalem]{ulem}
\usepackage[makeroom]{cancel}
\usepackage{subcaption}

\usepackage[iso-8859-7]{inputenc}

\newtheorem{introductiontheorem}{Theorem}
\newtheorem{theorem}{Theorem}[section]
\newtheorem{corollary}[theorem]{Corollary}
\newtheorem{lemma}[theorem]{Lemma}
\newtheorem{proposition}[theorem]{Proposition}
\newtheorem{remark}[theorem]{Remark}
\newtheorem{definition}[theorem]{Definition}
\newtheorem{example}[theorem]{Example}

\newtheorem{setup}[theorem]{Setup}

\def\red{\operatorname{red}}

\def\dim{\operatorname{dim}}
\def\codim{\operatorname{codim}}

\def\deg{\operatorname{deg}}

\def\det{\operatorname{det}}

\def\NE{\operatorname{NE}}
\def\NEb{\overline{\operatorname{NE}}}

\def\Pic{\operatorname{Pic}}
\def\Eff{\operatorname{Eff}}
\def\Mov{\operatorname{Mov}}
\def\Movb{\overline{\operatorname{Mov}}}
\def\Nef{\operatorname{Nef}}

\def\PGL{\operatorname{PGL}}

\def\det{\operatorname{det}}
\def\disc{\operatorname{disc}}

\def\res{\operatorname{res}}
\def\mod{\operatorname{mod}}

\newcommand{\x}{\chi}

\newcommand{\OO}{\mathcal{O}}
\newcommand{\II}{\mathcal{I}}

\renewcommand{\AA}{\mathbb{A}}
\newcommand{\NN}{\mathbb{N}}
\newcommand{\Ll}{\mathcal{L}}
\newcommand{\ZZ}{\mathbb{Z}}
\newcommand{\QQ}{\mathbb{Q}}
\newcommand{\RR}{\mathbb{R}}
\newcommand{\CC}{\mathbb{C}}
\newcommand{\PP}{\mathbb{P}}

\newcommand{\isom}{\cong}
\newcommand{\defeq}{\vcentcolon=}
\newcommand{\map}{\smash{\xymatrix@C=0.5cm@M=1.5pt{ \ar[r]& }}}
\newcommand{\rmap}{\smash{\xymatrix@C=0.5cm@M=1.5pt{ \ar@{-->}[r]& }}}
\newcommand{\psmap}{\smash{\xymatrix@C=0.5cm@M=1.5pt{ \ar@{..>}[r]& }}}

\newcommand{\overbar}[1]{\mkern 17mu\overline{\mkern-20mu #1 \mkern-20mu}\mkern 20mu}

\newcommand{\link}[1]{\hspace{-.04cm}\xleftrightarrow{\,\scalebox{0.5}{\emph{($#1$)}}\,}\hspace{-.04cm}}

\makeatletter
\newcommand{\myitem}[1]{%
	\item[#1]\protected@edef\@currentlabel{\emph{#1}}%
}
\makeatother

\author{Tiago Duarte Guerreiro, Sokratis Zikas}

\title{On Mori dreamness of blowups along space curves}

\date{\today}

\begin{document}
	
	\maketitle

	\begin{abstract}
		We study the problem of determining when the blowup $X \to \mathbb{P}^3$ along a smooth space curve $C$ is a Mori Dream Space. We obtain sufficient conditions, as well obstructions to the Mori dreamness of $X$ based on the external geometry of $C$. We furthermore find infinitely many pairs $(g,d)$ such that the corresponding Hilbert schemes $H_{g,d}^S$ admit components whose general element has these obstructions. As a consequence we show that Mori dreamness is not an open property in flat families and exhibit various degenerational pathologies.
	\end{abstract}

	
	\section{Introduction}	
	Mori Dream Spaces are a special kind of varieties introduced by Hu and Keel \cite{HuKeel} that behave optimally with respect to the Minimal Model Program. 
	From a birational point of view, they are special in the sense that one can run an MMP for any divisor, not just the canonical divisor, and the steps boil down to combinatorial data provided by the Mori chamber decomposition of its several cones.
	Typical examples of Mori Dream Spaces include toric and, more generally, Fano-type varieties.

	Despite their optimal behaviour with respect to the MMP, the property of being a Mori Dream Space, or \emph{Mori dreamness}, is not preserved under common operations such as blowups or taking hyperplane sections (see \cite{Mukai,castravettevelev} and \cite{jow2011lefschetz,ottem2015birational, denisi2024birational});
	\cite{CastravetMDSblowups} provides and excellent survey to the topic.
	Moreover the behaviour in families is not optimal either, as it is neither an open nor a closed property:
	for example Mukai \cite{Mukai}  shows that the blowup $X$ of $\PP^3$ along $9$ very general points is not a Mori Dream Space;
	however, by results of Castravet and Tevelev \cite{castravettevelev}, specializing these points to lie on a twisted cubic curve, $X$ becomes a Mori Dream Space, that is Mori dreamness is not an open property;
	we can then further degenerate to the example of Hasset and Tschinkel \cite{HassetTschinkelMDS}---$9$ points on the intersection of two plane cubic curves---so that $X$ is again not a Mori Dream Space, showing that Mori dreamness is not a closed property either.

	Since a normal $\QQ$-factorial variety with finitely generated Picard group of rank $1$ is trivially a Mori Dream Space, the problem becomes interesting starting from Picard rank $2$.  
	Already there the picture is wildly open though, as there is no clear classification even for smooth rational surfaces (see \cite{testa,zhouMDSWPS}). 
	In this paper we focus on the problem of determining Mori dreamness of the blowup of $\PP^3$ along a smooth curve $C$, based on the external geometry of $C$.
	
	The first results of this kind were those of Blanc and Lamy \cite{WeakFanos}: 
	motivated by the problem of constructing Sarkisov links, they give a complete list of pairs $(g,d)$ so that the blowup along a general element $C \in H_{g,d}^S$---the Hilbert scheme of smooth connected curves of genus $g$ and degree $d$ in $\PP^3$---is weak Fano, hence a Mori Dream Space.
	In this spirit, we show that the blowup along complete and almost complete intersection curves or curves in low degree surfaces is always a Mori Dream Space (see Propositions \ref{prop:aciMDS} and  \ref{prop:lowDegreeMDS}).
	In both cases we exhibit infinitely many pairs $(g,d)$ so that such curves span irreducible components of $H_{g,d}^S$, sometimes even several.
	
	Curiously the Hilbert schemes considered in \cite{WeakFanos} are all irreducible, with the exception of $(g,d)=(14,11)$.
	Even in this exceptional case, their result holds true for all components of $H_{14,11}^S$.
	This seems to be a coincidence due to the low numerics of these curves as, in Example \ref{ex:largeDegQuartic}, we show that $H_{141,35}^S$ admits two different components:
	general elements in one give rise to a Mori Dream Space, while in the other not.
	
	On the negative side of the story, to the best of our knowledge, prior to this work the only example of a curve whose blowup is \emph{not} a Mori Dream Space is that of \cite{Kuronya}:
	there  K{\"u}ronya provides an example of a curve of genus $159$ and degree $36$ contained in a quartic surface with a round cone of effective divisors;
	given the large numerics of the curve, the roundness of the cone reflects to the non-rational generation of the effective cone $\Eff(X)$ of $X$ (see Figure \ref{fig:NefQuartic} for a schematic description).
	Here we systematize this behaviour: 
	\begin{introductiontheorem}[Theorem \ref{thm:nonMDSquartics}]\label{mainthm:1}
		Let $(g,d)$ be integers satisfying $8g<d^2$ and let $C$ be a general element in $H_{g,d}^S$ contained in a quartic surface. 
		Define $r = d^2-8(g-1)$ and suppose that the generalized Pell equations 
		\[
		x^2 - r y^2 = -8 \quad \text{ and } \quad x^2 - r y^2 = 0
		\]
		do not admit any integer solutions, and that either $d \geq 16$ or $64 - 8d + 2g -2 \leq 0$.
		Denote by $X \to \PP^3$ the blowup along $C$.
		
		Then $\overline{\Mov}(X)$ has an irrationally generated extremal ray;
		in particular, $X$ is not a Mori Dream Space. 
	\end{introductiontheorem}
	Furthermore, in Examples \ref{ex:lowDegQuartic} and \ref{ex:largeDegQuartic} we provide infinitely many pairs $(g,d)$ so that a general element of some component of $H_{g,d}^S$ satisfies the hypotheses of Theorem \ref{mainthm:1}.
	This recasts the example of \cite{Kuronya} in a general setting, as the curves satistying the conditions imposed there span a divisor in our component constructed in Example \ref{ex:largeDegQuartic}.
	
	Our second obstruction to Mori dreamness utilizes heavily the, by now classical, theory of \emph{liaison} or \emph{linkage}.
	Starting from a curve $C'$ and performing a sufficiently general linkage $C' \link{S_1,S_2} C$, which we call \emph{super rigid linkage}, and blowing up $C$, we show that the strict transform of $C'$ spans an extremal ray of $\NE(X)$.
	If then $C'$ is \emph{very} general, i.e.\ not $\QQ$-canonical (see Definition \ref{def:Qcanonical}), then this ray is not contractible.
	More specifically we have:
	\begin{introductiontheorem}[Theorem \ref{thm:mainThmRigid} \& Corollary \ref{cor:degenToSemiample}]\label{mainthm:2}
		Let $(g',d')$ be a pair of integers different from $(3,4), (4,6)$, such that either $d' \geq 2g'-2$ or $\frac{3g'}{2} \leq d' \leq 2g'-2$.
		Let $n_1, n_2$ be sufficiently large integers and define $d = n_1n_2 -d'$ and $g = \frac{1}{2}(n_1 + n_2 - 4)(d - d') +g'$.
		
		Then there exists a subset $V_{g,d} \subset H_{g,d}^S$ that is a countable union of closed subsets, such that the blowup $X \to \PP^3$ along $C \in H_{g,d}^S$ is a Mori Dream Space if and only if $C \in V_{g,d}$;
		more specifically, if $C \not\in V_{g,d}$ then $X$ admits a nef but not semiample divisor.
	\end{introductiontheorem}
	\noindent This allows us to recover the non-openness of Mori dreamness in families.
	
	At this point it is worth comparing our results, with some similar flavoured results in the $K$-trivial world.
	In \cite[Corollary 1.3]{Matsushita} Matsushita proves that semiampleness is an open property for irreducible holomorphic symplectic manifolds (see also \cite[Theorem 3.8]{AM} for a global version). 
	On the other hand Corollary \ref{cor:degenToSemiample} shows that this is not true in our cases.
	Another obstruction that we did not manage to obtain is that of a variety having infinitely many Mori chambers.
	Given that the Picard rank of $X$ is $2$, such an example would have to include an infinite sequence of anti-flips from $X$, each decreasing the discrepancies and, in a sense, worsening the singularities.
	This would imply that no two of the corresponding models are isomorphic as abstract varieties.
	On the other hand, Hasset and Tschinkel \cite{HT10} provide an example of an irreducible holomorphic symplectic fourfold of Picard rank $2$, with infinitely many Mori chambers, where, however, every second model is abstractly isomorphic.
	
	The outline of the paper is as follows:
	in Section \ref{sec:preliminaries} we recall some basic facts about Mori Dream Spaces, as well as the theory of space curves and quartic surfaces;
	in Section \ref{sec:sufficient} we prove that the blowup of complete and almost complete intersection curves, as well as curves in low degree surfaces, give rise to Mori Dream Spaces;
	in Section \ref{sec:obstruction} we introduce and study curves with extremal surfaces and curves obtained by super-rigid linkage, and show that they actually are obstructions to Mori dreamness;
	in both Sections \ref{sec:sufficient} and \ref{sec:obstruction} we exhibit pairs $(g,d)$ where curves satisfying the corresponding conditions span components of their Hilbert Scheme;
	finally, in Section \ref{sec:skewLinkage} we define and develop the notion of \emph{super-rigid} and \emph{rigid skew linkage} which allows us to prove the \emph{if} part of Theorem \ref{mainthm:2} as well as exhibit some interesting examples;
	we chose to delay this section until the end as it is technical but quite independent from the previous material.

	\subsubsection*{Acknowledgements:} The authors would like to extend their heartfelt gratitude to Prof. St\'ephane Lamy for initial conversions on the topic of this paper as well as for the warm hospitality he provided at Institut de Math{\'e}matiques de Toulouse.
	The authors would also like to thank Francesco Denisi for very helpful comments on an earlier version.
	
	The first author would like to thank IMPA for the warm hospitality he received during the elaboration of this work. 
	He was supported by ERC StG Saphidir No. 101076412. 
	The sencond author would like to thank the laboratory IRL2924 Jean-Christophe Yoccoz for its hospitality.
	
	\subsection*{Notations and Conventions}
	We work over the field of complex numbers $\CC$.
	We denote by $H_{g,d}$ and $H_{g,d}^S$ the Hilbert scheme 
	of locally Cohen-Macaulay and smooth connected curves respectively, of arithmetic genus $g$ and degree $d$.
	Unless otherwise stated, a curve means a smooth connected closed subscheme of dimension $1$.
	Finally, for an $\RR$-vector space $V$ and $v,v_1,v_2 \in V$, we will write $v_1 \prec v \prec v_2$ if there exist positive real numbers $r_1,r_2$ so that $v = r_1v_1 + r_2v_2$.

	\section{Preliminaries}\label{sec:preliminaries}
	
	\subsection{Mori Dream Spaces}
	
	\begin{definition}\label{def:SQM}
		A \emph{small $\QQ$-factorial modification} ({SQM} for short) of a normal $\QQ$-factorial projective variety $X$ is a  birational map $f\colon X \psmap X'$, with $X'$ normal projective and $\QQ$-factorial, such that $f$ is an isomorphism in codimension $1$.
		
		An SQM $f\colon X \psmap X'$ is called a \emph{flip} if it sits in a diagram of the form
		\[
		\xymatrix@R=.4cm@C=.3cm{
			X \ar@{..>}[rr]^f \ar[rd]_p && X' \ar[ld]^q\\
			& Z
		}
		\]
		where $p$ and $q$ are small morphisms of relative Picard rank $1$.
		For a divisor $D \in \Pic(X)$ we say that $f$ is a \emph{$D$-positive/trivial/negative flip} if $D$ is ample/trivial/anti-ample over $Z$.
		If $p$ contracts the extremal ray $R = \RR_+[\gamma]$ of $\NE(X)$ we will say that $f$ \emph{flips} the ray $R$, or simply $\gamma$.	
	\end{definition}
	
	\begin{definition}\label{def:MDS}
		A normal projective variety $X$ will be called a \emph{Mori Dream Space} (MDS for short) if the following hold:
		\begin{enumerate}
			\item\label{it:MDS1} $X$ is $\QQ$-factorial and $\Pic(X)_{\QQ} = N^1(X)$;
			\item\label{it:MDS2} $\Nef(X)$ is the affine hull of finitely many semi-ample divisors;
			\item\label{it:MDS3} there is a finite collection of SQMs $f_i\colon X \psmap X_i$, such that each $X_i$ satisfies  \eqref{it:MDS2} and $\Mov(X)$ is the union of the $f_i^*(\Nef(X_i))$.
		\end{enumerate}
	\end{definition}
	
	The subcones $f_i^*(\Nef(X_i)) \subset \Mov(X)$ are called the \emph{Mori chambers} of $X$.
	
	\begin{remark}\label{rem:MDS}\leavevmode
		\begin{enumerate}
			\item\label{it:MDSrem1} By \cite[Remark 2.4]{Okawa} condition \eqref{it:MDS1} is equivalent to $\Pic(X)$ being finitely generated which in turn is equivalent to $\dim(\Pic^0(X)) = h^1(X,\OO_X) = 0$.
			In particular this is always satisfied when $X$ is a smooth rational variety.
			\item\label{it:MDSrem2} The SQMs of Definition \ref{def:MDS}\eqref{it:MDS3} are the only SQMs of $X$ (see \cite[Proposition 1.11(2)]{HuKeel}).
		\end{enumerate}
	\end{remark}
	
	\subsection{Linkage and the Hilbert-flag scheme}\label{subsec:linkageAndHilberFlag}

	\begin{definition}
		Let $S_1$ and $S_2$ be two smooth surfaces in $\mathbb P^3$ of degrees $n_1$ and $n_2$, respectively. 
		We say that the curves $C$ and $C'$ are \emph{$(S_1,S_2)$-linked} if 
		\[
		S_1 \cap S_2 = C \cup C'.
		\]
		We will write a linkage as $C \link{S_1,S_2} C'$.
		When the specific surfaces are irrelevant we will say that $C$ and $C'$ are \emph{$(n_1,n_2)$-linked} and denote it by $C \link{n_1,n_2} C'$.
	\end{definition}

	\begin{proposition}[{\cite[Subsection~8.4]{HarDef}}]\label{prop:gdLinked}
		Suppose that $C$ and $C'$ are $(n_1,n_2)$-linked.
		Then 
		\[
		d + d' = n_1n_2 \quad \text{ and } \quad g - g' = \frac{1}{2}(n_1 + n_2 - 4)(d - d'),
		\]
		where $(g,d)$ and $(g',d')$ the genera and degrees of $C$ and $C'$ respectively.
	\end{proposition}

	We recall the definition of a Hilbert-flag scheme originally introduced and studied by Kleppe in \cite{Kleppe, KleppeFlag}, in the case of space curves.
	
	\begin{definition}
		The \emph{Hilbert-flag scheme} $D_{g,d}(n_1,n_2)$ is the scheme parametrizing sequences $(C \subset V \subset \PP^3)$ with $C \in H_{g,d}$ and $V$ a complete intersection of type $(n_1,n_2)$ containing $C$.
	\end{definition}
	
	Given a point $(C,V) \in D_{g,d}(n_1,n_2)$ we always have a linkage $C \link{n_1,n_2} C'$;
	we can then naturally associate to it the point $(C',V) \in D_{g',d'}(n_1,n_2)$, where $g',d'$ are as in Proposition \ref{prop:gdLinked}.
	By \cite[Theorem 2.6]{KleppeFlag} this association gives us a natural isomorphism $\ell$ between the two Hilbert-flag schemes.
	
	Consider now the diagram
	\begin{equation}\label{eq:linkedFam}\tag{{\scriptsize$\leftrightarrow$}}
		\begin{tikzcd}[row sep=1.6em]
			D_{g,d}(n_1,n_2) \ar[r, "\ell"]  \ar[d, "p_1"']& D_{g',d'}(n_1,n_2) \ar[d, "p_1'"]\\
			H_{g,d} & H_{g',d'},
		\end{tikzcd}
	\end{equation}
	where $p_1$ and $p_1'$ denote the projections to the first factor.
	For a subscheme $U \subset H_{g,d}$ define $U' \defeq p_1'(\ell(p_1^{-1}(U)))$.
	We will call $U'$ the \emph{$(n_1,n_2)$-linked family} to $U$.
	
	The following statements show that, under certain assumptions, the process of \emph{linkage in families} of \eqref{eq:linkedFam} preserves openness and respects specializations.

	\begin{proposition}[{\cite[Proposition 3.8]{KleppeFlag}}]\label{prop:linkedOpen}
		Let $U, U'$ be as above and suppose that for all $C \in U$ we have
		\[
		h^1(\II_C(n_1-4)) = h^1(\II_C(n_2-4)) =0.
		\]	
		Then, if $U$ is open in $H_{g,d}$, so is $U'$ in $H_{g',d'}$.
	\end{proposition}
	
	\begin{proposition}[{\cite[Proposition 3.7]{KleppeFlag}}]\label{prop:specialization}
		Let $C_0 \in H_{g,d}$ and $C_0 \link{n_1,n_2} C_0'$ be a linkage.
		Suppose that $C_0$ is a specialization of $C \in  H_{g,d}$ so that
		\[
		h^0(\II_{C_0}(n_i)) = h^0(\II_{C}(n_i)) \quad \text{ and } \quad h^0(\OO_{C_0}(n_i-4)) = h^0(\OO_{C}(n_i-4)), 
		\]
		for $i = 1,2$.
		
		Then, there exists a linkage $C \link{n_1,n_2} C'$ so that $C_0'$ is a specialization of $C'$.
	\end{proposition}
	
	We now recall that standard notion of arithmetically Cohen-Macaulay curves.
	
	\begin{definition}\label{def:ACM}
		Let $C$ be a curve in $\mathbb P^3$. 
		We say that $C$ is \emph{arithmetically Cohen-Macaulay}, or ACM for short, if its \emph{Hartshorne-Rao module}
		\[
		HR(C) = \bigoplus_{m \in \mathbb Z} H^1\big(\mathbb P^3,\mathcal I_C(m)\big)
		\]
		is trivial.
	\end{definition}
	
	\begin{theorem}[{\cite[Theorem~2]{Ellingsrud}}]\label{thm:ACMareOpen}
		For any $(g,d)$ the locus of ACM curves in $H_{g,d}^S$ is open and smooth.
	\end{theorem}
	
	\begin{remark}\label{rem:linkedToACM}\leavevmode
		\begin{enumerate}
			\item\label{it:linkedToACM1} The locus of ACM curves in $H_{g,d}^S$ is not necessarily irreducible;
			its irreducible components are parametrised by an extra piece of data called the \emph{$h$-vector} (see \cite[8.11 and 8.12]{HarDef}).
			\item\label{it:linkedToACM2} Theorem \ref{thm:ACMareOpen} holds more generally for ACM subschemes of codimension $2$ in $\PP^n$.
			It is no longer true in codimension $3$ or more (see \cite[Remark 8.10.2]{HarDef}).
			\item\label{it:linkedToACM3} Linkage preserves the Hartshorne-Rao module up to grading shifts (cf.\ \cite[Subsection 8.6]{HarDef});
			in particular, if two curves are linked, then one is ACM if and only if the other is.		
			More generally Rao \cite{Rao} proves that two curves are connected by a series of linkages and deformations \emph{if and only if} their Rao modules differ by a grading shift.
		\end{enumerate}
	\end{remark}

	\subsection{$\QQ$-canonical curves}
	
	We now introduce the notion of a $\QQ$-canonical curve that will play a fundamental role in Section \ref{subsec:SRigidLinkage}. 
	
	\begin{definition}\label{def:Qcanonical}
		A curve $C \subset \PP^3$ is called \emph{$\QQ$-canonical} if a rational multiple of $H|_C$ is linearly equivalent to $K_C$, that is $\mu K_C \sim \nu H|_C$,  for some integers $\mu, \nu$.
		
		If $\mu = 1$ and $\nu \geq 1$ we say that $C$ is \emph{$\nu$-subcanonical}, or simply \emph{subcanonical}.
		A $1$-subcanonical curve is called \emph{canonical}.
	\end{definition}
	
	\begin{example}\label{ex:Qcan}\leavevmode
		\begin{enumerate}
			\item Any rational curve and, quite trivially, any elliptic curve is $\QQ$-canonical.
			\item\label{it:QcanCI} If $C$ is a smooth complete intersection of surfaces of degrees $n_1, n_2$ then, by adjunction we have
			\[
			K_C \sim (n_1 + n_2 - 4)H|_C,
			\]
			i.e.\ it is $(n_1 + n_2 - 4)$-subcanonical.
			In fact a classic theorem of Gherardelli \cite{Gherardelli} states a partial inverse:
			if $C$ is a subcanonical ACM curve, then it is a complete intersection.
		\end{enumerate}
	\end{example}
	
	Before we prove that $\QQ$-canonicity is a speciality condition, we recall the following classic theorem.
	
	\begin{theorem}[Clifford's Theorem, {\cite[Chapter III, \S 1]{ACGH}}]\label{thm:Clifford}
		Let $C$ be a curve of genus $g$ and $D$ an effective divisor of degree $d \leq 2g-1$.
		Then
		\[
		h^0(C,D) \leq \frac{d}{2} +1,
		\]
		with equality if and only if $D=0, K_C$ or $D$ is a multiple of a hyperelliptic divisor.
	\end{theorem}
	
	\begin{proposition}\label{prop:QcanonicalNonGen}
		Let $(g,d)$ be a pair of integers with $g\geq 2$ and $d \geq 1$.
		Suppose that either $d \geq 2g-2$ or $\frac{3g}{2} \leq d < 2g-2$ and $(g,d) \neq (3,4), (4,6)$.
		
		Then the locus of $\QQ$-canonical curves of genus $g$ and degree $d$ is a countable union of closed strict subsets of every component of $H_{g,d}^S$.	
	\end{proposition}
	
	\begin{proof}	
		We first prove that, for a fixed curve $C\in \mathcal{M}_g$ and an integer $d$, the locus of divisors $D \in \Pic^d(C)$ satisfying $\mu K_C \sim \nu D$,  for some integers $\mu, \nu$ is discrete.	
		Let $\frac{2g-2}{d} = \frac{n}{m}$, where $m,n$ are coprime and fix a divisor $D_0 \in \Pic^d(C)$ so that $mK_C \sim nD_0$.
		Then $\mu = km$ and $\nu = kn$ for some $k \geq 1$ and so we have
		\[
		\nu D \sim \mu K_C = km K_C \sim kn D_0 = \nu D_0
		\]
		i.e.\ $D_0$ and $D$ differ by a $\nu$-torsion divisor of degree $0$; the latter being discrete in $\Pic^0(C)$ proves our claim.

		We may now describe the locus of $\QQ$-canonical curves as
		\begin{equation}\tag{$\dagger$}\label{eq:Qcan}
			\mathcal{QK}_{g,d} \defeq \bigcup_{C\in \mathcal{M}_g}
			\left\{
			\begin{array}{c}
				\text{$B$ a basis of $V$}\\
				\text{up to scaling}
			\end{array}
			\middle|
			\begin{array}{c}
				V \in \operatorname{Gr}\big(4,H^0(C,D)\big), \, rD \sim_{\QQ}K_C \text{ for $r \in \QQ$}\\
				D \,\text{ is a very ample divisor of degree $d$} 
			\end{array}
			\right\}.
		\end{equation}
		This description shows immediately that the locus of $\QQ$-canonical curves is a countable union of closed subschemes.
		To complete the proof it thus suffices to show that $\mathcal{QK}_{g,d}$ is a strict subset of every component of  $H_{g,d}^S$.
		The result will follow from a dimension count: the dimension of every component of $H_{g,d}^S$ is at least $4d$ (see \cite[Theorem 12.1]{HarDef}), so it suffices to show that $\dim\mathcal{QK}_{g,d}<4d$.

		Since the subset of $\Pic^d(C)$ of divisors $D$ with $rD \sim_{\QQ}K_C$ for some $r \in \QQ$ is discrete, for the purposes of dimension counting, it suffices to choose any such $D$ such that $h^0(C,D)$ is maximal.
		We then get
		\[
		\dim\mathcal{QK}_{g,d} = (3g-3) + 4\big(h^0(C,D) - 4\big) + 15,
		\]
		where the numbers from left to right are: the dimension of $\mathcal{M}_g$, the dimension of the Grassmanian $\operatorname{Gr}\big(4,H^0(C,D)\big)$ and the dimension of the space of bases of a $4$-dimensional vector space adjusted for scaling.
		
		We now take cases based on $d$:
		If $d > 2g-2$ then any divisor $D \in \Pic^d(C)$ is non-special and so $h^0(C,D) = d -g +1$ which gives
		\[
		\dim\mathcal{QK}_{g,d} = (3g-3) + 4(d-g-3) + 15 = 4d - g < 4d.
		\]
		If $d = 2g-2$ then the divisor with maximal sections is $K_C$ for which we have
		\[
		\dim\mathcal{QK}_{g,d} = (3g-3) + 4(g-4) + 15 = 7g - 4 = 4(2g-2) - g + 4 = 4d - g +4,
		\]
		which again, if $g > 4$, is strictly less than $4d$.
		We are left with the case $g\leq 4$ corresponding to the pairs $(g,d) = (2,2), (3,4)$ and $(4,6)$: the latter two we have purposely excluded and the former leading to an empty $H_{g,d}^S$.
		
		We finally treat the case $d < 2g-2$.
		We will apply Theorem \ref{thm:Clifford} and, to do so optimally, we will further decompose $\mathcal{QK}_{g,d}$ into $\mathcal{QK}_{g,d}^{\operatorname{gen}} \sqcup \mathcal{QK}_{g,d}^{\operatorname{hyp}}$, where the latter two are defined as in \eqref{eq:Qcan} with the union ranging over the non-hyperelliptic curves and hyperelliptic curves respectively instead. 
		Before we proceed we remark that the locus of hyperelliptic curves in $\mathcal{M}_g$ has dimension $2g-1$:
		this can be easily computed by counting ramification points, using the Riemann-Hurwitz formula, up to the action of $\PGL_2(\CC)$.
		
		For the general case we have
		\[
		\dim\mathcal{QK}_{g,d}^{\operatorname{gen}} < 3g-3 + 4\left(\frac{d}{2} - 3\right) +15 = 2d + 3g,
		\]
		which is less than or equal to $4d$ when $d \geq \frac{3g}{2}$.
		As for the hyperelliptic case may take $D$ to be a multiple of a hyperelliptic divisor and obtain
		\[
		\dim\mathcal{QK}_{g,d}^{\operatorname{hyp}} = 2g-1 + 4\left(\frac{d}{2} - 3\right) +15 = 2d + 2g + 2
		\]
		which is strictly less than $4d$ when $d > g+1$.
		However, this is always satisfied when $d \geq \frac{3g}{2}$ and $g \geq 3$.
		When $g = 2$, $d < 2g-2 = 2$ in which case $H_{g,d}^S$ is empty.
	\end{proof}
	
	\begin{remark}
		The exclusion of $(g,d) = (3,4)$ and $(4,6)$ in Proposition \ref{prop:QcanonicalNonGen} is necessary:
		every curve of genus and degree $(3,4)$ and $(4,6)$ is a complete intersection of surfaces of degrees $1,4$ and $2,3$ respectively;
		they are thus $\QQ$-canonical by Example \ref{ex:Qcan} (both cases are actually canonical).	
	\end{remark}

	\subsection{Curves on quartic surfaces}\label{subsec:quartics}

	\begin{definition}\label{def:disc}
		Let $S \subset \PP^3$ be a smooth quartic surface with intersection matrix $Q_S$.
		We define the \emph{discriminant of $S$} to be the integer
		\[
		\disc(S) \defeq (-1)^{\rho-1}\det(Q),
		\]
		where $\rho$ is the Picard rank of $S$.	
		Note that, by the Hodge index theorem, $\disc(S)$ is a positive integer. 
	\end{definition}
	
	\begin{lemma}\label{lem:discriminant}
		Let $S \subset \PP^3$ be a smooth quartic surface of Picard rank $2$ and discriminant $r$.
		For any $C \in \Pic(S)$ we have
		\[
		4 C^2 = d^2 - rn^2
		\]
		for some integer $n$ and $d = H\cdot C$.
	\end{lemma}
	
	\begin{proof}
		By \cite[ex.\ 20.7, pg.\ 144]{HarDef} $\Pic(S)/\ZZ H$ has no torsion, and so we may choose a basis of $\Pic(S)$ of the form $\langle H,D \rangle$.
		Note that then
		\[
		r = (H. D)^2 - H^2 D^2 = (H. D)^2 - 4 D^2
		\]
		If $C = mH +nD$ we have
		\[
		\begin{aligned}
			4 C^2 &
			= 16m^2 + 8mn (H.D) + n^2(4 D^2)
			= 16m^2 + 8mn (H.D) + n^2((H.D)^2 - r)\\
			&= 16m^2 + 8mn (H.D) + n^2(H.D)^2 - rn^2
			= (4m + n(H.D))^2 - rn^2 = d^2 - rn^2.
		\end{aligned}
		\]
	\end{proof}
	
	Lemma \ref{lem:discriminant} gives us a direct computational way of checking whether a quartic surface of Picard rank $2$ and given discriminant admits curves of prescribed self intersection.
	
	\begin{corollary}\label{cor:Pell}
		Let $S \subset \PP^3$ be a smooth quartic surface of Picard rank $2$ and discriminant $r$ and consider the generalized Pell equations
		\[
		P_{r,k} : x^2 - r y^2 = -4k.
		\]
		If $P_{r,2}$ and $P_{r,0}$ admit no integer solutions then there are no rational or elliptic curves on $S$ respectively.
		
		In particular, if neither $P_{r,2}$ nor $P_{r,0}$ admit integer solutions, then $\NE(S)$ is not a closed cone and the extremal rays of $\NEb(S)$ are irrationally spanned.
	\end{corollary}
	
	\begin{proof}
		A rational curve $C$ on $S$ satisfies $C^2 = -2$;
		similarly for elliptic curves we have $C^2 = 0$.
		Therefore, the first part follows from readily from Lemma \ref{lem:discriminant}.
		
		As for the second part, if the two equations have no solutions, then every curve on $S$ has positive self-intersection.
		In particular the cone of curves coincides with the positive cone
		\[
		P(S) \defeq \big\{ z \in \Pic(S) \,\big|\, z^2>0,\, H\cdot z >0 \big\} \cup \{0\},
		\]
		which is clearly not a closed cone.
		
		Finally let $R$ be an extremal ray of $\NEb(S) = \overline{P}(S)$ and suppose by contraposition that $R$ is rational.
		Let $z$ be an integral class on $R$.
		Then $z^2 = 0$ which, by Lemma \ref{lem:discriminant}, would imply that there exists an integer solution to $P_{r,0}$, a contradiction. 
	\end{proof}

	\section{Sufficient Conditions for Mori dreamness}
	\label{sec:sufficient}
	
	\subsection{Complete and almost complete intersections}
	
	\begin{definition}
		A subscheme $Y \subset \PP^n$ is called an \emph{almost complete intersection} if its ideal $I_Y \subset \CC[x_0, \ldots, x_n]$ is generated by $\codim(Y) + 1$ polynomials.
	\end{definition}
	
	\begin{proposition}\label{prop:aciMDS}
		Let $C \subset \mathbb P^3$ be a smooth complete or almost complete intersection curve and $X$ be the blowup of $\mathbb P^3$ along $C$.  Then $X$ is a Mori Dream Space.
	\end{proposition}
	
	\begin{proof}
		We only treat the almost complete intersection case, the complete intersection being similar.
		Denote by $E$ be the exceptional divisor of the blowup 
		and let $H_1$ and $H_2$ be the pullback of general hyperplanes of $\PP^3$.
		Since $C$ is an almost complete intersection there exists homogenous polynomials $f_i$ so that $I_C = (f_1,f_2,f_3)$.
		Denote $S_i$ the strict transforms of the surfaces $\{f_i = 0\}$.
		Then
		\begin{enumerate}
			\item $\langle E,H_1,H_2 \rangle \cap \langle S_1,S_2,S_3 \rangle=\{ 0\}$  and
			\item $E\cap H_1 \cap H_2 = S_1 \cap S_2 \cap S_3 = \emptyset$.
		\end{enumerate}
		It follows from \cite[Theorem~1.3]{Ito} that $X$ is a Mori Dream Space.
	\end{proof}
	
	\begin{example}[Complete intersections]\label{ex:ci}
		Let $C \subset \mathbb P^3$ be a smooth complete intersection $\{f_1 = f_2 = 0\}$, where $f_i$ are homogeneous polynomials of degrees $n_1$ and $n_2$ respectively, where $n_1\leq n_2$.
		Then 
		\[
		d = n_1n_2 \quad \text{ and } \quad g = \frac{1}{2}n_1n_2(n_1 + n_2 - 4) + 1.
		\]
		By \cite[Exercise 1.3, pg.\ 7]{HarDef} complete intersection curves span an irreducible component of $H_{g,d}^S$.	
		
		The Mori chamber decomposition is particularly easy in this case. Indeed every movable divisor on $X$ is also nef and so the only SQM of $X$ is the identity.
		We have
		\[
		\Eff(X) = \langle E, S_1 \rangle \quad \text{ and } \quad \Mov(X) = \Nef(X) = \langle H , S_2 \rangle.
		\]
		the contraction given given by the linear system $|S_2|$ is a fibration to $\PP^1$ if $n_1 = n_2$;
		if $n_1 < n_2$ then it is divisorial contracting $S_1$; it is given by
		\[
		\begin{array}{ccc}
			\PP^3 & \rmap &  \PP(1,1,1,1,n_2-n_1)\\
			(x_0:\ldots:x_3) & \mapsto & \left(x_0f_1:x_1f_1:x_2f_1:x_3f_1:f_2\right),
		\end{array}
		\]
		whose image is the hypersurface 
		\[
		(zf_1-f_2 = 0) 
		\]
		of degree $n_2$, where $z$ is the variable of weight $n_2-n_1$.
	\end{example}
	
	\begin{example}[Almost complete intersections]\label{ex:almostCI}
		A systematic way of producing almost complete intersection curves is by performing a linkage starting from a complete intersection curve.
		In fact, \emph{every ACM} almost complete intersection is obtained like that (see \cite[Section 3]{liaison}).
		
		Let $C_0$ be a smooth complete intersection $(f_1 = f_2 =0)$ of degrees $n_1 \leq n_2$.
		Choose surfaces $S_i$ containing $C_0$ of degrees $m_1$ and $m_2$ with $m_1\leq m_2$, so that $S_i = (g_i = 0)$, with
		\[
		g_1 = \lambda_2 f_1 - \lambda_1 f_2 \quad \text{ and } \quad g_2 = \mu_2 f_1 - \mu_1 f_2.
		\]
		The curve $C$ obtained by the linkage $C \link{S_1,S_2} C_0$ is cut out by the $2\times 2$ minors of the matrix
		\[
		\begin{pmatrix}
			f_1 & \lambda_1 & \mu_1\\
			f_2 & \lambda_2 & \mu_2
		\end{pmatrix},
		\]
		that is $I_C = (g_1,g_2,\lambda_1\mu_2 - \lambda_2\mu_1)$.
		Furthermore we have
		\[
		d + d_0 = m_1m_2 \quad \text{ and } \quad g - g_0 = \frac{1}{2}(m_1 + m_2 - 4)(d - d_0),
		\]
		where $(g_0,d_0)$ denote the genus of degree of $C_0$.
		By Theorem \ref{thm:ACMareOpen} such curves form an open (and smooth) subset of the corresponding Hilbert scheme.	
	\end{example}
	
	\subsection{Curves in low degree surfaces}\label{subsec:curvesInLowDegreeSurf}
	
	We begin with an auxiliary lemma:
	
	\begin{lemma}\label{lem:crepantDoublePts}
		Let $S \subset \AA^3$ be a surface with a double point $p$
		and let $\pi\colon (U,S_U) \to (\AA^3,S)$ be the blowup at $p$.
		Then $\pi$ is $(\AA^3,S)$-crepant and $S_U$ has again at worst double points.
	\end{lemma}
	
	\begin{proof}
		This is a local calculation.
		We may assume that $p = (0,0,0)$ and $S = \{F=0\}$, with
		\[
		F = f_2(x_1,x_2,x_3) + \sum_{i=3}^n f_i(x_1,x_2,x_3),
		\]
		and $f_i \in \CC[x_1,x_2,x_3]$ are homogeneous of degree $i$;
		moreover there exists a neighbourhood $\AA^3 \isom U_0 \subset U$ so that $\pi$ is locally given by
		\[
		(y_1,y_2,y_3) \mapsto (y_1,y_1y_2,y_1y_3).
		\]
		Then $\pi^*F = y_1^2 F_U$, where
		\[
		F_U = f_2(1,y_2,y_3) + \sum_i^n y_1^if_i(1,y_2,y_3)
		\]
		and so $S_U = \{F_U=0\}$ has again at worst double points.
		Furthermore 
		\[
		\pi^*(K_{\AA^3} + S) = \pi^*(S) = S_U + 2E = K_U + S_U,
		\]
		i.e.\ $\pi$ is $(\AA^3,S)$-crepant.
	\end{proof}
	
	\begin{proposition}\label{prop:lowDegreeMDS}
		Let $C\subset\PP^3$ be a smooth curve contained in a surface $S$ of degree $n \leq 3$, and denote by $f\colon X \to \PP^3$ the blowup along $C$.
		
		Then, for any $0 \ll \delta <1$, $(X,\delta S)$ is a dlt log-Fano pair;
		in particular $X$ is a Mori Dream Space.	
	\end{proposition}
	
	\begin{proof}
		Assuming that $(X,\delta S)$ is dlt and log-Fano for some $0<\delta <1$, the result follows from \cite[Corollary 1.3.2]{BCHM}.
		Note that 
		\[
		-(K_X + \delta S) = (4-n\delta)H - (1-\delta)E = (4-n)H + (1-\delta)(n H - E);
		\]
		with $(4-n)H$ being the pullback of an ample divisor and $nH-E$ being $f$-ample, the sum is ample for a sufficiently large $\delta<1$.
		Thus we are only left with showing that $(X,\delta S)$ is dlt.
		Without loss of generality we may assume that $S$ is irreducible.
		We distinguish cases based on the singularities of $S$.
		
		First assume that $S$ has only isolated double points.
		Let $(X_1,S_1) \to (\PP^3,S)$ be an embedded log resolution obtained  by repeatedly blowing up along the singular points of $S$.
		Then, by Lemma \ref{lem:crepantDoublePts}, $(X_1,S_1) \to (\PP^3,S)$ is crepant.
		Denote by $p\colon (W,S_W) \to (X_1,S_1)$ the blowup of $X_1$ along the strict transform of $C$ with exceptional divisor $E_W$.
		We then have
		\[
		K_W = p^*K_{X_1} + E_W \quad \text{ and } \quad S_W = p^*S_1 - E_W,
		\]
		where $S_W$ denotes the strict transform of $S_1$.
		Combining we get that $p\colon (W,S_W) \to (X_1,S_1)$ is crepant too.
		Thus from the commutative diagram
		\[
		\xymatrix@R=.2cm@C=.2cm{
			& (W, S_W) \ar[ld]_p \ar[rd]^q\\
			(X_1, S_1)	\ar[rd] && (X, S) \ar[ld]\\
			& (\PP^3, S)
		}
		\]
		we deduce that $q$ is crepant.
		Finally $(W,S_W)$ being log smooth implies that $(X,S)$ is dlt and so is $(X,\delta S)$ for any $0\leq \delta \leq 1$.
		
		Assume now that $S$ has an isolated triple point, in which case $n=3$ and $S$ is a cone over a smooth plane cubic curve.
		Let $X_1 \to \PP^3$ be the blowup along the vertex of the cone with exceptional divisor $E_1$.
		Similarly to the previous case we get the commutative diagram 
		\[
		\xymatrix@R=.2cm@C=.2cm{
			& (W, \Delta_W) \ar[ld] \ar[rd]^q\\
			(X_1, \Delta_1)	\ar[rd] && (X, S) \ar[ld]\\
			& (\PP^3, S)
		}
		\] 
		where this time $\Delta_1 = S_1 + E_1$ and $\Delta_W$ is the strict transform of $\Delta_1$.
		Again we deduce that $q$ is crepant and $(W,\Delta_W)$ is log smooth, and we conclude similarly.
		
		Finally assume that $S$ does not have isolated singularities.
		Then $n=3$ and $S$ is double along a line $l$.
		Let $p\colon X_1 \to \PP^3$ be the blowup along $l$ with exceptional divisor $E_1$.
		This time we get a diagram
		\[
		\xymatrix@R=.4cm@C=.2cm{
			(W_1, \Delta_{W_1}) \ar[d]_p \ar@{..>}[rr]^{\chi}&&(W,\Delta_W) \ar[d]^q\\
			(X_1, \Delta_1)	\ar[rd] && (X, S) \ar[ld]\\
			& (\PP^3, S)
		}
		\] 
		where $\Delta_1 = S_1 + E_1$ and $\Delta_{W_1}$ is the strict transform of $\Delta_1$ under $p$, $\Delta_W = \chi_*(\Delta_{W_1})$ and $\chi$ is a $(K_{W_1}+\Delta_{W_1})$-trivial flop.
		Once again we deduce that $q$ is crepant and conclude similarly.
	\end{proof}

	\begin{example}[Curves on quadrics]\label{ex:curvesOnQuadrics}
		By \cite[IV, Theorem 6.4]{Hartshorne} if $C$ is a space curve of genus and degree $(g,d)$ with
		\begin{equation}\label{eq:gdQuadric}\tag{$\star$}
			\def\arraystretch{1.7}
			g = 
			\left\{
			\begin{array}{ll}
				\frac{1}{4}d^2  - d +1	& \text{ if $d$ is even}\\
				\frac{1}{4}(d^2-1)  - d +1& \text{ if $d$ is odd}
			\end{array}
			\right.
		\end{equation}
		then $C$ is contained in a quadric;
		the former case is realized by curves of bidegree $(a,a)$ while the latter of type $(a,a+1)$.
		Note that the latter ones are never complete intersections, they are however ACM curves (c.f.\ \cite[Exercise 8.1]{HarDef}).
		Combining with Proposition \ref{prop:lowDegreeMDS} we get that, for any $d$ and $g$ as in \eqref{eq:gdQuadric}, and $C \in H_{g,d}^S$ then $X$ is an MDS;
		if moreover $d$ is odd, then we are not in the setting of Example \ref{ex:ci}.
		
		More generally we can define the sets
		\[
		W_{(a,b)} \defeq 
		\overbar{
			\left\{ 
			[C] \in H^S_{g,d} \, \middle| \, 
			\begin{array}{c}
				C \subset Q \text{ smooth quadric}\\
				\text{with } \OO_Q(C) \isom \OO_Q(a,b)
			\end{array}
			\right\}
		}.
		\]
		By \cite[Proposition 4.11]{Nasu06}, if 
		\[
		d = a + b > 4 \quad \text{ and } \quad g = (a-1)(b-1) > 2d-8
		\]
		then $W_{(a,b)}$ is an irreducible component of $H^S_{g,d}$,
		and for any $[C] \in W_{(a,b)}$, $X_C$ is an MDS.
		Note that the two conditions are not necessary:
		for $(a,b) = (3,5)$ we have $(g,d) = (8,8)$;
		while $g \leq 2d-8$, $W_{(3,5)}$ coincides  with $H_{8,8}^S$.		
	\end{example}

	\begin{example}[Curves on cubics]\label{ex:curvesOnCubics}
		Similarly to the previous example, given a septuple of numbers $t = (k; m_1,\dots,m_6)$ satisfying
		\[\def\arraystretch{1.4}
		\left\{
		\begin{array}{l}
			k > m_1 \geq m_2 \geq \ldots \geq m_6 \geq 0, \quad k \geq m_1 + m_2 + m_3\\
			d = 3k - \sum_{i=1}^{6}m_i, \quad \text{ and } \quad g = \binom{k-1}{2} - \sum_{i=1}^{6}\binom{m_i}{2}
		\end{array}
		\right.
		\] 
		we may define the sets
		\[
		W_{t} \defeq 
		\overbar{
			\left\{ 
			[C] \in H^S_{g,d} \, \middle| \, 
			\begin{array}{c}
				C \subset T \text{ smooth cubic}\\
				\text{with } \OO_T(C) \isom \OO_T(t)
			\end{array}
			\right\},
		}
		\]
		which, for $d>9$, have dimension $\dim W_t = d +g +18$.
		
		By \cite[Lemma 4.3]{Nasu06}, if $d<12$ then $h^1(\II_C(3)) = 0$ while, if $d \geq 12$,
		\[
		h^1(\II_C(3)) = \#\{i\,|\, b_i = 2\} + 3\#\{i\,|\, b_i = 1\} + 6\#\{i\,|\, b_i = 0\}.
		\]
		If $t$ is so that $h^1(\II_C(3)) = 0$, then $H_{g,d}^S$ is smooth at every point of $W_t$.
		If $\dim W_t \geq 4d$, or equivalently $g > 3d-19$, then $W_t$ in an irreducible component of $H_{g,d}^S$ (see \cite[Proposition 4.4]{Nasu06}).
		On the other hand, if $h^1(\II_C(3)) = 1$ then $H_{g,d}^S$ is non-reduced at every point of $W_t$;
		still if $d>9$ and $g \geq 3d-18$ then $W_t$ is an irreducible component of $\big(H_{g,d}^S\big)_{\red}$ (see \cite[Theorem 1.2]{Nasu06}).
		
		Note that we can produce infinitely many types $t$, corresponding to different pairs $(g,d)$, satisfying any of the two conditions above:
		for example choosing
		\[
		t = (k; 3,3,3,3,3,3) \quad \text{ or } \quad t = (k; 3,3,3,3,3,2)
		\] 
		we land in the first and second cases respectively;
		then for $k$ sufficiently large, the conditions $d>9$ and $g > 3d-19$ are always satisfied.
	\end{example}

	\begin{remark}
		In both Examples \ref{ex:curvesOnQuadrics} and \ref{ex:curvesOnCubics} the genus of the curve grows quadratically with $(a,b)$ and $k$ respectively, while the degree grows linearly.
		Therefore the conditions $g> 2d-8$ and $g > 3d-19$ are always satisfied for sufficiently large pairs $(a,b)$ and $k$ respectively.
		In particular we have components of $H_{g,d}^S$ with $g$ and $d$ arbitrarily large, so that the blowup of $\PP^3$ along a general element in the component is an MDS. 
	\end{remark}

	\section{Components of $H_{g,d}^S$ with obstructions to Mori dreamness}
	\label{sec:obstruction}
	
	In this section we study two classes of space curves that allow us to retain some control on the various cones of $X$, their blowup $X\to \PP^3$: curves obtained by \emph{super-rigid linkage} and curves with \emph{extremal surfaces}.
	They both share a common feature: all external curves relevant to the birational geometry of $X$ lie on one surface.
	
	We begin this section with two auxiliary lemmas.
	
	\begin{lemma}\label{lem:nefBs}
		Let $X$ be a projective $\QQ$-factorial variety, $H$ an ample divisor and $S$ a hypersurface on $X$.
		Suppose $D_1, D_2$ are divisors so that $H \prec D_1 \prec S \prec D_2$.
		Then
		\begin{enumerate}
			\item\label{it:nef} for any curve $\gamma\subset X$, 
			\[
			D_1 \cdot \gamma \leq 0 \implies S \cdot \gamma < 0;
			\]
			in particular $\gamma \subset S$.
			\item\label{it:Bs} if $D_2$ is effective then it cannot contain $S$ in its base locus.
		\end{enumerate}
	\end{lemma}

	\begin{proof}\leavevmode	
		By assumption there exist $a_1, b_1 > 0$ so that $D_1 = a_1H + b_1S$.
		Since $H$ is ample $H\cdot \gamma >0$ for any curve $\gamma$;
		thus, if $D_1\cdot \gamma \leq 0$ for some curve $\gamma$, we have
		\[
		b_1 S\cdot\gamma <  (a_1H + b_1S)\gamma = D_1 \cdot \gamma \leq 0 \implies S\cdot\gamma <0.
		\]
		This is \eqref{it:nef}.
		
		As for \eqref{it:Bs}, if $D_2$ is effective and contains $S$ in its base locus then, for some integer $k$, $D_3 \defeq D_2 - kS$ is effective and does not contain $S$ in its base locus.
		But then $H \prec D_2 \prec D_3$ and so $D_2 = a_2H + b_2D_3$ for some $a_2, b_2 > 0$.
		Finally $H$ being ample and $D_3$ not containing $S$ in its base locus, neither does $D_2$; a contradiction.
	\end{proof}
	
	Lemma \ref{lem:nefBs} implies the following in the $2$-dimensional case:
	
	\begin{lemma}\label{lem:zeroCurvesOnBoundry}
		Let $S$ be a smooth surface and $C \subset S$ a curve with $C^2\leq 0$.
		Then, for any ample divisor $H$ and any $\epsilon >0$, $C-\epsilon H$ is not effective; 
		equivalently $C$ lies on the boundary of $\NEb(S)$.
	\end{lemma}
	
	\begin{proof}
		Note that $H \prec C \prec C-\epsilon H$ and $(C-\epsilon H)\cdot C <0$.
		Assuming that $C-\epsilon H$ is effective, $C$ is contained in the base locus of $C-\epsilon H$, contradicting Lemma \ref{lem:nefBs}\eqref{it:Bs}.
	\end{proof}
	
	\subsection{Super-rigid linkage}\label{subsec:SRigidLinkage}

	\begin{definition}
		Let $C,C' \subset \PP^3$ be smooth space curves and $S_1, S_2$ smooth surfaces.
		A linkage $C\link{S_1,S_2} C'$, with $\deg(S_i) = n_i$ and $n_1\leq n_2$, is called \emph{$C$-super rigid} if 
		\[
		2g' - 2 - (n_i-4)d' < 0,
		\]
		for $i =1,2$, where $(g',d')$ is the genus and degree of $C'$.
	\end{definition}
	
	The significance of the integers $2g' - 2 - (n_i-4)d'$ is explained in the following Lemma:
	
	\begin{lemma}\label{lem:intersectionsSRigid}
		Let $C \link{S_1,S_2} C'$ be a linkage, where $S_1, S_2$ are smooth surfaces of degrees $n_1,n_2$ in $\mathbb P^3$.
		Denote by $X \to \PP^3$ the blowup of $\PP^3$ along $C$. 
		
		Then,on $X$, we have
		\[
		S_i \cdot C' = (C')^2_{S_j} = 2g' - 2 - (n_j-4)d'.
		\]
	\end{lemma}
	
	\begin{proof}
		We have
		\[
		S_i \cdot C' = S_i|_{S_j}\cdot C' = (C')^2_{S_j} = 2g' -2 - K_{S_j}\cdot C' = 2g' -2 - (n_j-4)d',
		\]
		where the last two equalities follow from two applications of the adjunction formula.
	\end{proof}

	\begin{proposition}\label{prop:conesSRigid}
		Let $C\link{S_1,S_2} C'$ be a $C$-super rigid linkage with $\deg(S_i) = n_i$ and $n_1 \leq n_2$, and let $X \to \PP^3$ be the blowup along $C$.
		Then
		\[
		\Eff(X) = \langle E, S_1 \rangle \quad \text{ and } \quad \Mov(X) = \langle H, S_2 \rangle.
		\]	
		Furthermore a divisor $D$ is nef if and only if $D \cdot C' \geq 0$;
		equivalently $C'$ generates an extremal ray of $\NE(X)$.
	\end{proposition}
	
	\begin{proof}
		The inclusions $\langle E, S_1 \rangle \subseteq \Eff(X)$ and $\langle H, S_2 \rangle \subseteq \Mov(X)$ are clear;
		similarly if a divisor $D$ is nef, then $D\cdot C' \geq 0$.
		We now prove all the reverse inclusions.
		
		Let $D \equiv r(\delta S_2 - \epsilon H)$ be an effective divisor, where $\epsilon,\,\delta>0$.
		Then, by Lemma \ref{lem:intersectionsSRigid}, $D\cdot C' < 0$ and so $C'$ is contained in the support of $D$.
		Consequently $D$ corresponds to a section of $\II_{C\cup C'}$, which is a complete intersection ideal of the surfaces $S_1,S_2$.
		Thus $S_2\prec D \prec S_1$, and in particular $D \in \langle E,S_1\rangle$, which shows the first equality.
		
		As for the second equality assume that $D$ is movable;
		in particular, by Lemma \ref{lem:nefBs}\eqref{it:Bs}, $S_1$ is not contained in the base locus of $D$, thus $D|_{S_1}$ is effective.
		On the other hand 
		\[
		D|_{S_1} = r(\delta S_2-\epsilon H)|_{S_1} = r(C' - \epsilon H|_{S_1}).
		\]
		Since $(C')_{S_1}^2 = 2g' - 2 - (n_1-4)d' < 0$    and $H|_{S_1}$ is ample we obtain a contradiction by Lemma \ref{lem:zeroCurvesOnBoundry}.
		This shows that $\Mov(X) = \langle H, S_2 \rangle$.
		
		For the last part, first note that the nef cone, being the closure of the ample cone, will be contained in the (closure of the) movable cone $\langle H , S_2 \rangle$.
		Let $D \in \langle H , S_2 \rangle$ so that $D \equiv aH + bS_2$ for $a,b \geq 0$.
		However $S_2$ is nef away from $C'$, thus $D$ can only fail to be nef on $C'$.
	\end{proof}
	
	\begin{proposition}\label{prop:semiAmp=>Qcan}
		Let $C\link{S_1,S_2} C'$ be a $C$-super rigid linkage with $\deg(S_i) = n_i$ and $n_1 \leq n_2$, $X \to \PP^3$ be the blowup along $C$, and $D$ be a nef divisor on $X$ that is trivial against $C'$.
		
		Then $\OO_{C'}(kD) \isom \OO_{C'}$ for some $k\geq 1$ if and only if $C'$ is $\QQ$-canonical;
		in particular, if $D$ is semiample, then $C'$ is $\QQ$-canonical.
	\end{proposition}
	
	\begin{proof}
		By Proposition \ref{prop:conesSRigid} we have $D \sim aH + bS_2$ for some $a,b \geq 0$.
		Restricting $D$, first on $S_1$ and then on $C'$ we obtain
		\[
		D|_{S_1} \sim (aH + bS_2)|_{S_1} = aH|_{S_1} + bC' \implies D|_{C'} \sim aH|_{C'} + bC'|_{C'}.
		\]
		By adjunction on $C' \subset S_1$ we have that
		\[
		C'|_{C'} = K_{C'} - K_{S_1}|_{C'} = K_C - (n_1-4)H|_{C'}.
		\]
		Combining everything we get
		\begin{equation}\label{eq:restriction}\tag{$\perp$}
			D|_{C'} = (a-b(n_1-4))H|_{C'} + bK_C.
		\end{equation}
		Now if $\OO_{C'}(kD) \isom \OO_{C'}$ then, by \eqref{eq:restriction}, $C'$ is $\QQ$-canonical.
		Conversely if $C'$ is $\QQ$-canonical then, up to possibly scaling \eqref{eq:restriction} by some integer $k$, we obtain that $k(a-b(n_1-4))H|_{C'} \sim -kbK_C$, i.e.\ $\OO_{C'}(kD) \isom \OO_{C'}$.
	\end{proof}
	
	\begin{theorem}\label{thm:mainThmRigid}
		Let $(g',d')$ be integers such that either $d' \geq 2g'-2$ or $\frac{3g'}{2} \leq d' \leq 2g'-2$ and $(g',d') \neq (3,4), (4,6)$.
		Let $n_1, n_2$ be integers, sufficiently large so that
		\[
		h^1(\II_{C'}(n_i-4)) = 0 \quad \text{ and } \quad 2g'-2 -(n_i-4)d' <0
		\]
		for all $C' \in H_{g',d'}^S$, and define $d = n_1n_2 -d'$ and $g = \frac{1}{2}(n_1 + n_2 - 4)(d - d') +g'$.
		
		There exists a subset $U_{g,d} \subset H_{g,d}^S$ that is the complement of a countable union of closed subschemes so that, for any $C \in U_{g,d}$, the blowup $X \to \PP^3$ along $C$ admits a nef integral divisor, that is not semiample;
		in particular $X$ is not a Mori Dream Space.
	\end{theorem}
	
	\begin{proof}
		Let $U_{g',d'} \subset H_{g',d'}^S$ be the locus of non $\QQ$-canonical curves.
		By Proposition \ref{prop:QcanonicalNonGen}, $U_{g',d'}$ can be written as the intersection of countably many open an dense subsets $V'_n$, $n\in \NN$.
		Define $V_n$ to be the $(n_1,n_2)$-linked family to $V'_n$ of \eqref{eq:linkedFam} and $U_{g,d}$ to be the intersection of the $V_n$.
		Then, by Proposition \ref{prop:linkedOpen}, every $V_n$ is open in $H_{g,d}^S$ and every element in $C \in U_{g,d}$ is obtained by a $C$-super rigid linkage from an element $C' \in U_{g',d'}$.
		By Proposition \ref{prop:semiAmp=>Qcan} there exists a nef divisor $D$ on $X$ that is not semiample.
		We deduce that $X$ is not a Mori Dream Space by Definition \ref{def:MDS}\eqref{it:MDS2}. 
	\end{proof}
	
	The motto of Theorem \ref{thm:mainThmRigid} is: \emph{being obtained by a general linkage is an obstruction to Mori dreamness.}

	\begin{example}
		While Theorem \ref{thm:mainThmRigid} is stated asymptotically, in practice it is very explicit.
		For example let $(g',d') = (2,5)$ and $C' \in H_{2,5}^S$ be a non-$\QQ$-canonical curve, which is a generality condition by Proposition \ref{prop:QcanonicalNonGen}.
		Any such $C'$ is $(2,3)$-linked to a line, therefore, by Remark \ref{rem:linkedToACM}\eqref{it:linkedToACM3}, $C'$ is an ACM curve.
		Then the conditions
		\[
		h^1(\II_{C'}(n_i - 4)) = 0 \quad \text{ and } \quad 2g'-2 -(n_i-4)d'<0
		\]
		are satisfied for any $n_1,n_2 \geq 5$.
		For instance, for $n_1=n_2=5$, $(g,d) = (47,20)$ and all conditions of Theorem \ref{thm:mainThmRigid} are satisfied.
		
		Certain examples with the starting curve $C'$ not being ACM can also be treated, where $h^1(\II_{C'}(n_i-4))$ can be computed on a case by case basis.
	\end{example}

	\subsection{Curves with extremal surfaces}\label{subsec:extremalSurfaces}

	\begin{proposition}\label{prop:conesExtremal}
		Let $C$ be a smooth space curve contained in a surface $S \in \mathbb P^3$.
		Denote by $X \to \PP^3$ the blowup along $C$.
		Suppose that the class $S|_S$ is not in the interior of $\NEb(S)$.
		
		Then $\Eff(X) = \langle E, S \rangle$ and we have
		\[
		\Movb(X) \isom \NEb(S) \cap H_\vdash \quad \text{ and } \quad \Nef(X) \isom \Nef(S) \cap H_\vdash,
		\]
		where $H_\vdash = \big\{mH|_S + nC \in \Pic(S) \,\big|\, m\geq 0 \big\}$ and the isomorphism is given by restriction.
	\end{proposition}

	\begin{proof}
		Since $E$ and $S$ are effective divisors the cone they span is clearly a subcone of $\Eff(X)$.
		Suppose by contraposition that $\Eff(X) = \langle E, D \rangle$, with $D \equiv S - \epsilon H$.
		By Lemma \ref{lem:nefBs}\eqref{it:Bs} $S$ is not contained in the stable base locus of $D$ and so $D|_S$ is effective.
		However we have
		\[
		D|_S \sim S|_S - \epsilon H|_S \implies S|_S \sim D|_S + \epsilon H|_S.
		\]
		$S|_S$ lying on the boundary of $\NEb(S)$ and $D|_S, H|_S$ being effective imply that both $D|_S$, but more importantly, $H|_S$ also lie on the boundary of $\NEb(S)$, which is absurd.
		
		Once again, given a movable or nef divisor $D \equiv \alpha H + \beta S$, with $\alpha, \beta >0$, its restriction on $S$ has to be effective/nef and so we get the two inclusions
		\[
		\Movb(X)|_S \subseteq  \NEb(S) \cap H_\vdash
		\quad \text{ and } \quad
		\Nef(X)|_S \subseteq \Nef(S) \cap H_\vdash,
		\]
		respectively.
		So we are left with proving the reverse inclusions.
		
		Let $\gamma$ be a curve in the stable base locus of $D$.
		Then $D\cdot \gamma \leq 0$ and, since $\gamma$ cannot be proportional to a fiber of $X \to \PP^3$, $H\cdot \gamma > 0$ which implies that $S \cdot \gamma < 0$, and consequently $\gamma \subset S$.
		Since the class of $D|_S$ lies in the interior of $\NEb(S)$, it can only be trivial against a finite number of curves.
		Consequently there can only be finitely many curves in the stable base locus of $D$, i.e.\ $D$ is movable.
		Arguing similarly, if $D$ is negative against some curve $\gamma$ then $\gamma \subset S$ and thus $D|_S\cdot \gamma < 0$.
		Thus, if the restriction $D|_S$ of $D$ on $S$ is nef, then $D$ is nef itself.
	\end{proof}
	
	Given a curve $C$ and a surface $S$ as in Proposition \ref{prop:conesExtremal}, we call $S$ the \emph{extremal surface of $C$.}

	\begin{figure}[h!]
		\begin{tikzpicture}
			\node[inner sep=0cm] (Cone) at (0,0)
			{\includegraphics[width=.3\textwidth]{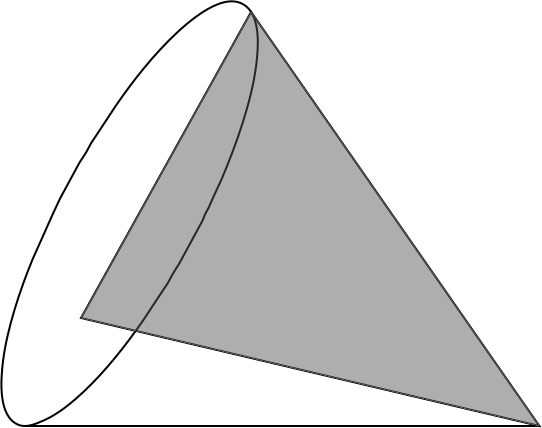}};

			\node[inner sep=1pt] (X) at (.4,-.45)
			{$\Nef(X)$};
			
			\node[inner sep=1pt] (S) at (1,1.4)
			{$\Nef(S)$};
			
			\node[inner sep=1pt, label={[label distance=-8pt]175:$C$}] (Cbul) at (-1.6,-.87)
			{{\small $\bullet$}};
			
			\node[inner sep=1pt, label={[label distance=-8pt]175:$H$}] (Hbul) at (-.58,.95)
			{{\small $\bullet$}};
		\end{tikzpicture}	
		\caption{The nef cone of $X$ as a slice of that of $S$.}
		\label{fig:NefQuartic}
	\end{figure}
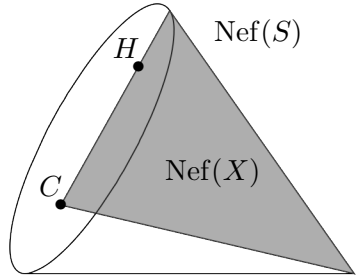
	
	\begin{lemma}\label{lem:notInterior}
		Let $C$ be a smooth space curve of genus and degree $(g,d)$ contained in a smooth surface $S$ of degree $n$ and denote by $H$ the hyperplane class
		
		If either $d \geq n^2$ or $S$ does not contain any curves of non-positive self intersection and $n^3 - (3n-4) d + 2g - 2 \leq 0$ then $S|_S \sim nH-C$ is not in the interior of $\NEb(S)$; 
		that is $S$ is the extremal of $C$.
	\end{lemma}

	\begin{proof}
		By adjunction formula we have $K_S = (n-4)H|_S$ and $K_C = (K_S +C)|_C$, which combined give 
		\[
		C^2 = \deg(K_C) - K_S\cdot C  = 2g-2 - (n-4)d.
		\]	
		We then may calculate that
		\[
		H \cdot (nH - C) = n^2 - d \quad \text{ and } \quad (n H - C)^2 = n^3 - (3n-4) d + 2g - 2.
		\]
		
		Since $H$ is an ample divisor, if $n^2-d \leq 0$ then clearly $nH - C$ cannot lie in the interior of $\NEb(S)$.
		Suppose now that $S$ does not contain any curves of non-positive self intersection and $n^3 - (3n-4) d + 2g - 2 \leq 0$.
		Then the interior of $\NEb(S)$ coincides with the closure of the positive cone
		\[
		P(S) \defeq \big\{ z \in \Pic(S) \,|\, z^2>0 \text{ and } H|_S\cdot z >0 \big\};
		\]
		since $(nH - C)^2 \leq 0$, it cannot lie in the interior of $\NEb(S)$.
	\end{proof}

	While the condition of not containing any curves of non-positive self intersection is hard to verify in general, in light of Corollary \ref{cor:Pell}, it is viable to do so for quartic surfaces of Picard rank $2$.
	Therefore for the rest of the section we stick to curves contained in such quartics. 
	
	In $H_{g,d}^S$ we define the locus 
	\[
	Q_{g,d}^{\circ} \defeq \Big\{ [C] \in H_{g,d}^S \, \Big| \, C \text{ is contained in a smooth quartic surface} \Big\}
	\]
	and denote by $Q_{g,d}$ its closure.
	By \cite{MoriK3} $Q_{g,d}$ is non-empty if an only if $8g<d^2$.
	
	\begin{theorem}\label{thm:nonMDSquartics}
		Let $(g,d)$ be integers satisfying $8g<d^2$ and let $C$ be a general element in $Q_{g,d}$. 
		Define $r = d^2-8(g-1)$ and suppose that the equations $P_{r,2}$ and $P_{r,0}$ of Corollary \ref{cor:Pell} do not admit any integer solutions, and that either $d \geq 16$ or $64 - 8d + 2g -2 \leq 0$.
		Denote by $X \to \PP^3$ the blowup along $C$.
		
		Then $\overline{\Mov}(X)$ has an irrationally generated extremal ray;
		in particular, $X$ is not a Mori Dream Space. 
	\end{theorem}
	
	\begin{proof}
		By \cite{MoriK3}, since $8g<d^2$, $Q_{g,d}$ is non-empty and for a general $C \in Q_{g,d}$ we may choose a smooth quartic surface $S$ containing $C$ with $\Pic(S) = \langle H, C \rangle$, whose discriminant is $r$.
		By assumption $P_{r,2}$ and $P_{r,0}$ do not admit any integer solutions and so, by Corollary \ref{cor:Pell}, $S$ does not contain any curves of non-positive self intersection and $\NEb(S)$ is spanned by two irrational rays. 
		
		Recall that the class of $S$ in $N^1(X)$ is $4H-E$ whose restriction on itself is equivalent to $4H - C$.
		Then, by Lemma \ref{lem:notInterior}, $S|_S$ is not in the interior of $\NEb(S)$ and, by Proposition \ref{prop:conesExtremal}, $\Movb(X)$ is a slice of $\NEb(S)$ by a half-plane.
		Therefore one of the rays of $\Movb(X)$ is irrationally generated.
	\end{proof}

	We now verify that there exists infinitely many pairs $(g,d)$ satisfying the assumptions of Theorem \ref{thm:nonMDSquartics}, so that the loci $Q_{g,d}$ are actually irreducible components of $H_{g,d}^S$.

	\begin{proposition}\label{prop:componentsQuartics}
		Let $(g,d)$ be any pair of integers with $8g < d^2$ and define $r = d^2 - 8(g-1)$.
		Suppose further that $d > 16$, $64 - 8d + 2g -2  \geq 0$ and $P_{r,2}$ does not admit any integer solutions.
		
		Then $Q_{g,d}$ is an irreducible component of $H_{g,d}^S$ of dimension $33 +g$.
	\end{proposition}
	
	\begin{proof}
		Since $8g < d^2$ by \cite[Theorem 1]{MoriK3} $Q_{g,d}$ is non-empty;
		furthermore, for $C \in Q_{g,d}^{\circ}$ and a general quartic $S$ containing $C$, $\Pic(S)$ is spanned by the classes $H|_S$ and $C$.
		Then the discriminant of $S$ is $r$ and, by Corollary \ref{cor:Pell} and $P_{r,2}$ not admitting integer solutions, $S$ does not contain any rational curves.
		By \cite[Theorem 1]{Kleppe}, to show that $Q_{g,d}$ is an irreducible component of $H_{g,d}^S$, it suffices to show that $h^1(\II_C(4)) = 0$.
		However, since $d > 16$, there exists a unique quartic containing it and the exact sequence
		\[
		0 \to H^0(\II_C(4)) \to H^0(\OO_{\PP^3}(4)) \to H^0(\OO_C(4)) \to H^1(\II_C(4)) \to 0
		\]
		yields
		\[
		h^1(\II_C(4)) = h^0(\II_C(4)) - h^0(\OO_{\PP^3}(4)) + h^0(\OO_C(4)) =  h^0(\OO_C(4)) -34,
		\]
		reducing to showing that $h^0(\OO_C(4)) = 34$.
		
		Let $S$ be the unique quartic containing $C$.
		We then have
		\begin{equation}\label{eq:sesK3}
			0 \to H^0(S,\OO_S(4H-C)) \to H^0(S,\OO_S(4)) \to H^0(C, \OO_C(4)) \to H^1(S,\OO_S(4H-C)), 
		\end{equation}
		the latter being isomorphic to $H^1(S,\OO_S(C-4H))$ by Serre duality.
		We may calculate that 
		\[
		(C-4H)^2 = 64 - 8d + 2g -2  \geq 0 \quad \text{ and } \quad H\cdot (C-4H) = d- 16 > 0;
		\]
		since $S$ contains no rational curves, $\OO_S(C-4H)$ has no fixed locus and consequently $h^1(S,\OO_S(C-4H)) = 0$.
		Since $H\cdot (4H-C) = 16-d < 0$, $h^0(S,\OO_S(4H-C)) = 0$ and, by Riemann-Roch $h^0(S,\OO_S(4)) = 34$.
		Finally \eqref{eq:sesK3} yields $h^0(C, \OO_C(4)) =34$ completing the proof of the first part.
		
		As for the dimension of $Q_{g,d}$ we use the facts that: 
		we have an embedding $L \hookrightarrow \Pic(S)$ where $L$ is the sublattice spanned by $H|_S$ and $C$;
		$S$ is the unique quartic containing $C$;
		on $S$, $C$ moves in a $g$-dimensional linear system.
		Combining everything we obtain
		\[
		\dim Q_{g,d} = (20 - 2) + 15 + g = 33+g,
		\]
		where the numbers from left to right are: the dimension of $L$-polarized K3 surfaces, the choice of a basis of $H^0(S,H|_S)$ up to scaling and finally the dimension of the linear system $\PP(H^0(S,C))$.	
	\end{proof}
	
	In \cite{KleppeOttem} the authors use similar techniques to construct non-reduced components of $H_{g,d}^S$ starting from curves on quartics or quintics.
	Their chosen surfaces though always contain a rational curve and therefore does not fit our requirements of Theorem \ref{thm:nonMDSquartics}.
	
	\begin{example}[Components of low degree]\label{ex:lowDegQuartic}
		The pairs $(g,d) = (3,9), (7,10), (15,12)$ and $(23,14)$ constitute all pairs with $d < 16$ satisfying the assumptions of Theorem \ref{thm:nonMDSquartics} and $Q_{g,d}$, the locus of curves contained in a smooth quartic surface, is a component of $H_{g,d}^S$.
		
		For $(g,d) = (3,9)$ or $(7,10)$, by \cite{Ein} since $g +3 \leq d$, $H_{g,d}^S$ is irreducible.
		Standard computations using the Riemann-Roch formula (cf.\ \cite[Lemma 2.3]{WeakFanos}) reveal that every $C \in H_{g,d}^S$ is contained in a quartic surface.
		Therefore $Q_{g,d} = H_{g,d}^S$ and for a general $C \in Q_{g,d}$, the blowup $X \to \PP^3$ along $C$ is not a Mori Dream Space.
		However there exist special elements in $Q_{g,d}$ contained in cubics:
		e.g.\ the loci $W_t$ of Example \ref{ex:curvesOnCubics} for $t = (4; 1,1, 1, 0, 0, 0)$ and $(6;2, 2, 2, 1, 1, 0)$ respectively.
		The blowup along the latter producing Mori Dream Spaces by Proposition \ref{prop:lowDegreeMDS}. 
		
		For $(g,d) = (23,14)$ and $C = C_{23,14}$ a general element in $H_{g,d}^S$ we have the sequence of linkages
		\[
		C_{23,14} \link{4,5} C_{3,6} \link{3,3} C_{0,3} \link{2,2} C_{0,1},
		\]
		the last curve being a line and hence ACM, so is $C$ by Remark \ref{rem:linkedToACM}.
		By Theorem \ref{thm:ACMareOpen}, $Q_{g,d}$ is a component of $H_{g,d}^S$.
		On the other hand, by Example \ref{ex:curvesOnCubics}, the locus $W_t$ for $t = (11;4, 4, 3, 3, 3, 2)$, is a codimension $1$ subset of $Q_{g,d}$.
	\end{example}

	\begin{example}[Components of large degree]\label{ex:largeDegQuartic}
		For any $n \geq 7$ the pairs $(g,d) = (20n +1 ,5n)$ satisfy the assumptions of both Proposition \ref{prop:componentsQuartics} and Theorem \ref{thm:nonMDSquartics}.
		Note that $g = 4d+1$ and so we have 
		\[
		d \geq 35 \quad \text{ and } \quad 64 - 8d + 2g -2 = 64.
		\]
		Thus we only need to verify that the Pell equations $P_{r,2}$ and $P_{r,0}$ do not admit integer solutions.	
		We have $r = d^2 - 8(g-1) = d(d-32)$.
		Since $d$ is a multiple of $5$, reducing $P_{r,2}$ modulo $5$ we obtain
		$x^2 = 2(\mod 5)$, which has no integer solutions as $2$ is not a square modulo $5$.
		A solution to $P_{r,0}$ is equivalent to $r=  d(d-32)$ being a square, which we readily see is not, by considering a factorization of $d$ into primes.
		
		Thus, for such pairs, $H_{g,d}^S$ has a component $Q_{g,d}$ so that the blowup $X$ of $\PP^3$ along a general $C \in Q_{g,d}$ is not a Mori Dream Space.
		For all small values of $n$ we were able to find curves on $H_{g,d}$ lying on smooth cubic surfaces (for example one may use the method \texttt{typeOnCubic} on Macaulay2 \cite{M2} from the second author's website \cite{code}); 
		such curves, by Proposition \ref{prop:lowDegreeMDS}, give rise to Mori Dream Spaces.
		However we were unable to find a closed formula for a type $(k;m_1,\dots,m_6)$ corresponding to $(g,d) = (20n +1 ,5n)$.
		
		Curiously, for $n = 7$, we have $(g,d) = (141,35)$.
		For $t = (22; 6, 6, 6, 6, 4, 3)$ the locus $W_t \subset H_{141,35}^S$ of Example \ref{ex:curvesOnCubics} is an irreducible component.
		Thus $H_{141,35}^S$ admits two components: a general element in one giving rise to Mori Dream Space, while in the other not. 
	\end{example}

	\begin{remark}
		The results of Subsections \ref{subsec:curvesInLowDegreeSurf} and \ref{subsec:extremalSurfaces} seem to suggest that properties of the surface $S$ of minimal degree containing a curve $C$ reflect properties of the blowup $X$ of $\PP^3$ along $C$.
		While true under further assumptions (e.g.\ $\deg(S)\leq 3$), it is not true in a vacuum.
		Indeed, choosing $S$ arbitrary and slicing it with any general hypersurface of degree strictly greater than $\deg(S)$ we obtain a complete intersection curve whose surface of minimal degree containing it is $S$.
		No matter the geometry of $S$, $X$ is a Mori Dream Space by Proposition \ref{prop:aciMDS}.
	\end{remark}

	In this section we have mostly worked with curves on quartic surfaces of Picard rank $2$.
	The reasoning is twofold: on the one hand we want to provide statements for open subsets of the various Hilbert schemes and, for a general quartic surface $S$ containing a given general curve, the Picard rank of $S$ is $2$, the minimum possible when the curve is not a hypersurface section of $S$;
	on the other hand, the statements of subsection \ref{subsec:quartics} give us an easy way of checking when the various cones of $S$, and consequently of $X$, are non rationally generated.
	Regarding the latter, even when $\rho(S) \geq 3$, the behaviour of the various 
	cones still reduces to lattice theoretic arguments.
	
	In fact, to our knowledge, the first example of a space curve $C_0$ whose blowup is not a Mori Dream Space was obtained in \cite{Kuronya} and it relied on the fact that $C_0$ is contained in quartic surface of Picard rank $3$ with open cone of curves.
	The genus and degree of $C_0$ are $(159,36)$.
	For a general curve with these numerics contained in a quartic surface $S$, then $\rho(S)=2$ and $r = \disc(S) = 32$.
	We may verify that the Pell equations $P_{r,2}$ and $P_{r,0}$ do not admit any solutions and so, by Proposition \ref{prop:componentsQuartics}, $Q_{159,36}$ is an irreducible component of $H_{159,36}^S$, that is, $C_0$ is simply a special element in $Q_{159,36}$.

	We should finally remark that the statements of subsection \ref{subsec:quartics} still hold true for smooth surfaces of higher degree.
	Start from any curve $C$ and take $S$ to be a smooth surface of minimal degree containing $C$ then, depending on a case by case study, one can use linkage and Lefschetz-type theorems (cf.\ \cite[Theorem II.3.1]{Lopez}) to control the Picard rank of $S$.
	Moreover, the notion of discriminant of a surface of Definition \ref{def:disc} and Lemma \ref{lem:discriminant} still hold true for arbitrary smooth surfaces of Picard rank $2$.
	The main obstacle is the lack of a uniform lower bound for the self intersection of curves on a surface $S \subset \PP^3$ of fixed degree;
	in turn this reflects to a lack of finiteness of the number of Pell equations to be checked.

	\section{Rigid and super-rigid skew linkage} 
	\label{sec:skewLinkage}
	
	In this last part we prove a a partial inverse to Proposition \ref{prop:semiAmp=>Qcan}.
	We also relax the notion super-rigid linkage of subsection \ref{subsec:SRigidLinkage}, by allowing the curve $C'$ to to be a union of pairwise skew curves, which also allows us to exhibit some interesting behaviours.
	
	\subsection{Definition and basic properties}
	
	\begin{definition}
		A linkage $C\link{S_1,S_2} \Gamma$ is called \emph{skew} if $\Gamma$ is a configuration  of smooth, pairwise skew curves $\gamma_i$, $i \in I = \{1,\dots,m\}$.
		
		A skew linkage $C\link{S_1,S_2} \Gamma$ is called \emph{$C$-super rigid} (resp.\ \emph{rigid}) if 
		\[
		e_i \defeq 2g_i - 2 - (n_1-4)d_i < 0 \, \text{(resp.\ $\leq 0$)},
		\]
		for all $i \in I$, where $(g_i,d_i)$ denotes the genus and degree of $\gamma_i$ and $n_1 \leq n_2$ the degrees of $S_1,S_2$ respectively.
		If moreover  $n_1 = n_2 = n$ we say that $C\link{S_1,S_2} \Gamma$ 
		is \emph{balanced}.
	\end{definition}
	
	Our choice of naming is explained in Remark \ref{rem:naming}.
	
	\begin{setup}\label{set:rigidLinkage}
		Up to reordering the curves $\gamma_i$ we will always assume that $\frac{e_i}{d_i}$ form an increasing sequence.
		Furthermore we partition $I$ into a subsets $I_1, \dots, I_k$ so that for $i \in I_a$, $j \in I_b$
		\[
		\frac{e_i}{d_i} < \frac{e_j}{d_j} \iff a < b.
		\]
	\end{setup}
	
	The significance of the numbers $e_i$ and the partition $I = I_1 \sqcup \ldots \sqcup I_k$ are explained in the following lemma.
	
	\begin{lemma}\label{lem:skewLinkageInters}
		Let $C \link{S_1,S_2 } \Gamma$ be a skew linkage and let $X \to \PP^3$ be the blowup along $C$.
		
		Then 
		\[
		S_1 \cdot \gamma_i \leq S_2 \cdot \gamma_i = (\gamma_i)^2_{S_1} = e_i.
		\]
		Furthermore, for $i\in I_a$ and $j \in I_b$, we have
		\[
		f \prec \gamma_j \preceq \gamma_i \iff a \leq b;
		\]
		in particular $\gamma_i, \gamma_j$ are proportional if and only if $a = b$.
	\end{lemma}

	\begin{proof}
		For the proof of the first part is similar to that of Lemma \ref{lem:intersectionsSRigid}.
		As for the second part for any $k\in I$, since $S_2 \equiv n_2 H - E$, we may compute that
		\[
		E\cdot \gamma_k = (n_2H - S_2)\gamma_k = n_2d_k - e_k
		\]
		and so $\gamma_k \equiv d_k(l - n_2f) + e_kf$.
		We thus have
		\[
		d_i\gamma_j - d_j\gamma_i \equiv (d_ie_j - d_je_i)f \iff d_i\gamma_j \equiv (d_ie_j - d_je_i)f + d_j\gamma_i,
		\]
		where the coefficient $(d_ie_j - d_je_i)$ is non-negative if and only if $\frac{e_j}{d_j} \geq \frac{e_i}{d_i}$, i.e.\ if and only if $a\leq b$.
	\end{proof}

	Verbatim we also obtain the counterpart to Proposition \ref{prop:conesSRigid}.
	
	\begin{proposition}\label{prop:conesRigidLinkage}
		Let $C \link{S_1,S_2} \Gamma$ be a skew $C$-rigid linkage and let $X \to \PP^3$ be the blowup along $C$.
		Then
		\[
		\Eff(X) = \langle E, S_1 \rangle \quad \text{ and } \quad \Mov(X) = \langle H, S_2 \rangle.
		\]	
		Furthermore a divisor $D$ is nef if and only if $D \cdot \gamma_i \geq 0$ for all $1\leq i \leq m$.
	\end{proposition}
	
	In view of Lemma \ref{lem:skewLinkageInters} we will denote by $R_a$ the ray in $N_1(X)$ spanned by any curve $\gamma_i$ with $i \in I_a$.
	We further denote by $D_a$ the unique numerical class, up to scaling, in $N^1(X)$ that is perpendicular to $R_a$.
	We then obtain Figures \ref{fig:NE(X)} and \ref{fig:Eff(X)}.
	
	Our goal is to show that, under some assumptions on the curves $\gamma_i$ and the linkage $C \link{S_1,S_2} \Gamma$, $X$ is a Mori Dream Space whose Mori chamber decomposition is given by the divisors $D_a$.
	
	\begin{figure}
		\centering
		\captionbox{$\NE(X)$ with the rays spanned by the $\gamma_i$, the shaded region being the $D_a$-negative half space.\label{fig:NE(X)}}[.47\textwidth]
		{
			\begin{tikzpicture}
				\node[inner sep=5pt] (0) at (0,0)
				{\resizebox{1.5mm}{!}{$\bullet$}};
				\node[inner sep=5pt] (f) at (3,0)
				{$f$};
				\node[inner sep=5pt] (l) at (2.12,2.12)
				{$l$};
				\node[inner sep=5pt] (k) at (0,3)
				{$R_k$};
				\node[inner sep=5pt] (dts) at (-.3, 1.65)
				{\rotatebox{8}{$\dots$}};
				\node[inner sep=5pt] (a) at (-1.12, 2.78309)
				{$R_a$};
				\node[inner sep=5pt] (dts2) at (-.95, 1.4)
				{\rotatebox{35}{$\dots$}};
				\node[inner sep=5pt] (2) at (-2.24, 1.9956)
				{$R_2$};
				\node[inner sep=5pt] (1) at (-2.8, 1.07703)
				{$R_1$};
				\node[inner sep=5pt] (neg) at (-2.5, 3.1)
				{\scalebox{.8}{$[D_a< 0]$}};
				
				\begin{scope}		
					\clip (-3.1,0) rectangle ++(3.1,3.3);
					\draw[draw opacity=0, fill=black,fill opacity=.15, rotate around={22:(0,0)}] (0,0)rectangle ++(-3,5);
				\end{scope}
				\draw[-,line width=.5mm] (0,0)--(f);
				\draw[-] (0,0)--(l);
				\draw[-] (0,0)--(k);
				\draw[-] (0,0)--(a);
				\draw[-] (0,0)--(2);
				\draw[-,line width=.5mm] (0,0)--(1);
			\end{tikzpicture}
		}\hspace{.05\textwidth}
		\captionbox{$\Eff(X)$ with the shaded and striped regions signifying the movable and nef cones respectively.\label{fig:Eff(X)}}[.47\textwidth]{
			\begin{tikzpicture}
				\node[inner sep=5pt] (0) at (0,0)
				{\resizebox{1.5mm}{!}{$\bullet$}};
				\node[inner sep=5pt] (E) at (-3,0)
				{$E$};
				\node[inner sep=5pt] (H) at (-2.12,2.12)
				{$H$};
				\node[inner sep=5pt] (1) at (0,3)
				{$D_1$};
				\node[inner sep=5pt] (2) at (.7, 2.94727)
				{$D_2$};
				\node[inner sep=5pt] (dts) at (.7, 1.65)
				{\rotatebox{160}{$\dots$}};
				\node[inner sep=5pt] (k) at (1.68, 2.48548)
				{$D_k$};
				\node[inner sep=5pt] (S2) at (2.24, 1.9956)
				{$S_2$};
				\node[inner sep=5pt] (S1) at (2.8, 1.07703)
				{$S_1$};
				
				\begin{scope}
					\clip (-3,0) rectangle (current bounding box.north east);
					\clip (0,0) let \p1 = ($(2.5,0) - (0,0)$) in circle ({veclen(\x1,\y1)});
					\draw[-,line width=.5mm] (0,0)--(E);
					\draw[-] (0,0)--(H);
					\draw[-] (0,0)--(1);
					\draw[-] (0,0)--(2);
					\draw[-] (0,0)--(k);
					\draw[-] (0,0)--(S2);
					\draw[-,line width=.5mm] (0,0)--(S1);
					\fill[fill=black, fill opacity=.15] (H.center)--(0,5)--(S2.center)--(0,0);
					\fill[pattern=north west lines, fill opacity=.15] (H.center)--(0,5)--(1.center)--(0,0);
				\end{scope}
			\end{tikzpicture}
		}
	\end{figure}

	\subsection{Two contraction statements}
	
	\begin{lemma}\label{lem:generalVanishing}
		Let $X$ be a projective variety and $S$ a reduced hypersurface whose connected components $S_1,\dots, S_k$ are all Cartier divisors.
		Let $A$ be any Cartier divisor on $X$ and $n\in \NN$ so that $h^i(X,A)=0$ and $h^i(S_j,A + rS) = 0$ for all $i>0$, $1\leq j\leq k$ and $1 \leq r \leq n$.
		
		Then $h^i(X,A+rS)= 0$ for all $0 \leq r \leq n$.
	\end{lemma}
	
	\begin{proof}
		We first remark that, since the $S_j$ are pairwise disconnected, we have that $\OO_{S_j}(S_i) = \OO_{S_j}$ when $i\neq j$ and $\OO_{S_j}(S)$ otherwise.
		We will proceed by induction on $r$, the base case being satisfied by assumption.
		The structure sequence for $S_1$ induces the following sequence:
		\begin{align*}
			\ldots \to H^i(X,A+rS) \to H^i(X,A+rS + S_1) &\to H^i(S_1,A+rS + S_1) \to \ldots
		\end{align*}
		The right hand being equal to $H^i\big(S_1,A+(r+1)S\big)$ is zero by our initial assumptions and so is the left hand by the inductive hypothesis; thus $H^i(X,A+rS + S_j) = 0$.
		Similarly, from the structure sequence for $S_2$ we get
		\begin{align*}
			\ldots \to H^i(X,A+rS + S_1) \to H^i(X,A+rS + S_1 + S_2) &\to H^i(S_2,A+rS + S_1 + S_2) \to \ldots
		\end{align*}
		Once more the right hand side is equals $H^i\big(S_1,A+(r+1)S\big)$ and is zero by our initial assumptions and by the previous step so is the left hand side.
		We thus conclude that $H^i(X,A+rS + S_1 + S_2) = 0$.
		Repeating the argument for $S_3,\dots, S_k$ we deduce that 
		\[
		H^i(X,A+rS + S_1 + S_2 + \ldots + S_k) = H^i\big(X,A+(r+1)S\big) = 0,
		\]
		proving the inductive step.
	\end{proof}
	
	\begin{proposition}\label{prop:contractionNeg}
		Let $X$ be a projective threefold, $A$ an ample divisor and $S_1,S_2$ irreducible hypersurfaces so that $\OO_X(S_i)$ are Cartier.
		Suppose that $N$ is a nef but not ample divisor and that we have $A \prec N\prec S_2 \preceq S_1$.
		Write $C_1, \dots, C_k$ for the connected components of $C \defeq S_1 \cap S_2$, and assume that 
		\[
		S_2\cdot C_i < 0, \quad 2g_i -2 + S_1\cdot C_i < 0 \quad \text{ and } \quad h^0(C,\OO_{C}(rN)) \neq 0,
		\]
		for some $r \gg 0$, where $g_1$ denotes the arithmetic genus of $C_i$.
		
		Then $N$ is semi-ample.	
	\end{proposition}
	
	\begin{proof}
		Since $A$ is ample and $A \prec N\prec S_2 \prec S_1$, any curve that is zero against $N$ is negative against the $S_i$;
		it is therefore contained in both, i.e.\ it is one of the $C_i$.
		Thus it suffice to show that a multiple of $N$ does not have any base locus on the $C_i$.
		We will show that the restriction homomorphisms
		\[\setlength{\arraycolsep}{2pt}
		\begin{array}{ccccc}
			H^0(X,N) &\overset{\res}{\longrightarrow} &H^0(S_1,N) &\to &H^1(X,N-S_1)\\
			H^0(S_1,N) &\overset{\res}{\longrightarrow} &H^0(C,N) &\to &H^1(S_1,N-S_2)
		\end{array}
		\]
		are both surjective, or equivalently that $h^1(X,N-S_1) = h^1(S_1,N-S_2) =0$.
		Since $h^0(C,\OO_C(N)) \neq 0$, $\OO_{C_i}(N)$ is either ample or trivial; in any case, up to a multiple, it has no base points.
		
		Up to possibly scaling $A$ and $N$ we may assume that $h^1(X,A) = h^1(S_1,A) = 0$ and that there exists integers $m, k, l$ so that
		\[
		N = mA + kS_1 \quad \text{ and } \quad S_1 = lS_2 - A;
		\]
		we therefore have $N = (m - k)A + lkS_2$, which implies that $m-k > 0$.
		Note that $h^1(X,N-S_1) = h^1(X,mA + (k-1)S_1)$ and thus, by Lemma \ref{lem:generalVanishing}, it suffices to show that
		\[
		h^1(S_1, mA + r_1S_1) =0,
		\]
		for $1\leq r_1 \leq k-1$.
		However we have
		\begin{align*}
			mA +r_1S_1 = mA + r_1(lS_2 - A) = (m-r)A + lr_1S_2 = (m-k + k-r_1)A + lr_1S_2
		\end{align*}
		and, once again by Lemma \ref{lem:generalVanishing}, it suffices to show that 
		\[
		h^1(C_i, (m-r_1)A + r_2 S_2) = 0
		\]
		for all $1\leq r_2 \leq lr_1$.
		Since $\deg \OO_{C_i}(S_2) < 0$ it suffices to prove that $h^1(C_i, (m-r_1)A + lr_1 S_2)$ vanishes.
		Finally, repacking the divisor, we have
		\begin{align*}
			(m-r_1)A + lr_1 S_2 = (m-r)A + r_1(S_1 +A) = mA +r_1S_1
		\end{align*}
		and, once again because $\deg \OO_{C_i}(S_1) < 0$ and $r_1 \leq k -1$, it suffices to prove vanishing for the divisor $mA + (k-1)S_1 = N - S_1$,
		for which we get
		\[
		\deg \OO_{C_i}(N - S_1) \geq \deg \OO_{C_i}(S_1)^{\vee} > 2g_i -2;
		\]
		$\OO_{C_i}(N - S_1)$ being non-special we conclude.
		
		The vanishing of $H^1(S_1,N-S_2)$ is obtained similarly:
		we have $N-S_2 = (m-k)A + (lk-1) S_2$ and so, by Lemma \ref{lem:generalVanishing}, it suffices to show that 
		\[
		h^1(C_i,(m-k)A + r_2S_2) = 0
		\]
		for all  $1\leq r_2 \leq lk-1$;
		$\deg \OO_{C_i}S_2$ being negative, it suffices to prove vanishing for $r_2 = lk-1$.
		Repacking the divisor we get $(m-k)A + (lk-1) S_2 = N-S_2$ for which we have
		\[
		\deg\OO_{C_i}(N-S_2) \geq \deg\OO_{C_i}(S_2)^{\vee} \geq \deg\OO_{C_i}(S_1)^{\vee} > 2g_i -2,
		\]
		i.e.\ it is non-special, concluding the proof.
	\end{proof}
	
	In Proposition \ref{prop:contractionNeg} the assumption that $S_2\cdot C_i <0$ is crucial in order to deduce that $N \prec S_2$ and be able to use Lemma \ref{lem:generalVanishing}.

	\begin{proposition}\label{prop:contractionZero}
		Let $X$ be a projective threefold, $A$ an ample divisor and $S_1,S_2$ irreducible hypersurfaces so that $\OO_X(S_i)$ are Cartier.
		Suppose that $S_2$ is nef but not ample, $A\prec S_2 \preceq S_1$ and that 
		\[
		h^1(X,S_2-S_1) = 0, \quad h^1(S_1,\OO_{S_1}) = 0 \quad \text{ and } \quad h^0(C,\OO_C(S_2)) \neq 0,
		\]
		where $C \defeq S_1 \cap S_2$.
		
		Then $S_2$ is semiample.
	\end{proposition}
	
	\begin{proof}
		If $S_2 \sim S_1$ then any curve in the base locus of $S_2$ would be contained in $S_1$ too.
		If on the other hand $A\prec S_2 \prec S_1$, since $A$ is ample, any curve in the base locus of $S_2$ would be zero against it, thus negative against $S_1$.
		In any case, it suffices to show that $S_1$ has no base curves along $S_1$.
		
		Both restriction homomorphisms
		\[\setlength{\arraycolsep}{2pt}
		\begin{array}{cccccc}
			H^0(X,S_2) &\overset{\res}{\longrightarrow} &H^0(S_1,S_2) &\to &H^1(X,S_2-S_1)&\\
			H^0(S_1,S_2) &\overset{\res}{\longrightarrow} &H^0(C,S_2) &\to &H^1(S_1,S_2-C) &\isom H^1(S_1,\OO_{S_1}) 
		\end{array}
		\]
		are surjective since we assume that $h^1(X,S_2-S_1) = h^1(S_1,\OO_{S_1}) = 0$.
		Since $h^0(C,\OO_C(S_2)) \neq 0$, for any component $C_i$ of $C$ so that $S_2 \cdot C_i =0$, we have $\OO_{C_i}(S_2) \sim \OO_{C_i}$.
		Thus $S_2$ does not contain any of the curves that is zero against in its base locus.
		Therefore passing to a multiple of it we may get rid of any isolated base points, proving our claim.
	\end{proof}

	\subsection{SQMs on complete intersections}
	
	\begin{lemma}\label{lem:contracitonDecompBundle}
		Let $C$ be a smooth curve $\Ll_1,\Ll_2$ line bundles on $C$ and $P = \PP_C(\Ll_1 \oplus \Ll_2)$ together with the projection $r$ to $C$.
		Assume that $\Ll_1 \oplus \Ll_2^{\vee}$ is anti-ample and denote by $\sigma_i\colon C \to P$ the sections corresponding to $\Ll_i \to  \Ll_1 \oplus \Ll_2$, with image $\gamma_i$.
		
		We then have $\gamma_1^2 < 0, \gamma_2^2 > 0$ and $\gamma_1\cdot \gamma_2 = 0$.
		Moreover, these imply that $\gamma_2$ is a base point free divisor whose corresponding contraction contracts $\gamma_1$.
	\end{lemma}
	
	\begin{proof}
		We will heavily utilize the formula
		\[
		\OO_P(\gamma_i) = r^*\Ll_j^{\vee} \otimes \OO_P(1),
		\]
		as well as the fact that $\sigma_i^*$ and $r_*$ have the same effect on $\gamma_i$, both found in \cite[V, Proposition 2.6]{Hartshorne}.
		
		By the correspondence between sections and injections to $\Ll_1 \oplus \Ll_2$ we have that $\sigma_i^*\OO_{P}(1) = \Ll_i$.
		We now may compute that
		\[
		\gamma_1^2 = \deg\left(\OO_P(\gamma_1)|_{\gamma_1}\right) =
		\deg\left(r^*\Ll_2^{\vee}|_{\gamma_1} \otimes \OO_{\gamma_1}(1)\right) = \deg\left(\Ll_1 \otimes\Ll_2^{\vee}\right) < 0.
		\]
		In a similar fashion we may compute that
		\[
		\gamma_1 \cdot \gamma_2 = \deg(\Ll_1^{\vee} \otimes \Ll_1) = 0 \quad \text{ and } \quad \gamma_2^2 = \deg(\Ll_1^{\vee} \otimes\Ll_2) > 0
		\]	
		Since $\gamma_2\cdot f = 1$, $\gamma_2$ is a nef divisor that is zero precisely on $\gamma_1$ and $\gamma_1$ is not contained in the base locus of $\gamma_2$, we deduce that $\gamma_2$ is base point free.
	\end{proof}
	
	\begin{lemma}\label{lem:blpNormalBundle}
		Let $U$ be a smooth quasi-projective variety and $\gamma$ a codimension $2$ complete intersection of smooth hypersurfaces $S_1$, $S_2$.
		Denote by $r\colon V \to U$ the blowup of $U$ along $\gamma$ with exceptional divisor 
		\[
		F \isom \PP(\OO_{\gamma}(S_1) \oplus \OO_{\gamma}(S_2)).
		\]
		
		Then, if we denote by $\sigma_i$ the section induced by $\OO_{\gamma}(S_i) \hookrightarrow \OO_{\gamma}(S_1) \oplus \OO_{\gamma}(S_2)$, we have
		\[
		N \isom \OO_{\gamma}(S_i-S_j) \oplus \OO_{\gamma}(S_j)
		\]
		where $N$ denotes the normal bundle of $\sigma_i$ in $V$ and we identify $\sigma_i$ with $\gamma$ via $r$.
	\end{lemma}

	\begin{proof}
		We have the equalities
		\[
		\tilde{S_i} = r^*S_i - F  \implies
		\left\{
		\begin{array}{l}
			F = r^*S_j - \tilde{S_j}\\
			\tilde{S_i} = r^*(S_i-S_j) + \tilde{S_j}
		\end{array}
		\right.
		\]
		Notice that $\sigma_i$ is the complete intersection of $F$ with $\tilde{S_i}$ and $\OO_V(\tilde{S_j})|_{\sigma_i} \isom \OO_{\sigma_i}$.
		This gives
		\begin{align*}
			N &= \OO_{\sigma_i}(\tilde{S_i}) \oplus \OO_{\sigma_i}(F) = \OO_{\sigma_i}(r^*(S_i-S_j) + \tilde{S_j}) \oplus \OO_{\sigma_i}(r^*S_j - \tilde{S_j})\\
			&=\OO_{\sigma_i}(r^*(S_i-S_j)) \oplus \OO_{\sigma_i}(r^*S_j) \isom \OO_{\gamma}(S_i-S_j) \oplus \OO_{\gamma}(S_j).
		\end{align*}
	\end{proof}
	
	\begin{corollary}\label{cor:blowupUntilBalanced}
		Let $U$ be a smooth quasi-projective threefold and $\gamma$ a smooth curve that is a complete intersection of surfaces $S_1,S_2$.
		Assume further that $\OO_{\gamma}(S_i)$ are proportional in $\Pic(\gamma)$.
		
		Then there exists a finite sequence of blowups $(V_i,F_i) \to (V_{i-1},\gamma_{i-1})$, $i = 1,\dots, k$ over $\gamma$ so that:
		$(V_0,\gamma_0) = (U,\gamma)$;
		$F_i \to \gamma_{i-1}$ is a decomposable $\PP^1$-bundle with minimal section $\gamma_i$;
		the normal bundle of $\gamma_k \in V_k$ is balanced.
	\end{corollary}
	
	\begin{proof}
		Since $\OO_{\gamma}(S_i)$ are proportional, there exists a divisor $D \in \Pic(\gamma)$ so that $\OO_{\gamma}(S_i) \isom \OO_{\gamma}(a_iD)$;
		without loss of generality assume that $a_i$ are positive with $a_1 \geq a_2$.
		For any $a \in \ZZ$ we will abbreviate $\OO_{\gamma}(aD)$ with $\OO_{\gamma}(a)$.
		
		Applying the Euclidean algorithm for $a_1,a_2$ we obtain natural numbers $a_i, m_i$ so that
		\[
		\left\{
		\begin{array}{l}
			a_1 = m_1 a_2 + a_3\\
			a_2 = m_2 a_3 + a_4\\
			\dots \\
			a_n = m_n a_{n+1}.
		\end{array}
		\right.
		\] 
		We will obtain the desired sequence of blowups by first blowing up $\gamma$ and then, on every subsequent step blowing up the minimal section of the previous exceptional divisor.
		We describe the first step as all the subsequent ones are similar:	
		Let $r_1\colon (V_1,F_1) \to (U,\gamma)$ be the blowup of $U$ along $\gamma$ with exceptional divisor $F_1 \isom \PP(\OO_{\gamma}(a_1)\oplus \OO_{\gamma}(a_2))$.
		Then, by Lemma \ref{lem:blpNormalBundle}, the normal bundle $N_1$ of the minimal section $\gamma_1$ of $F_1$ in $V_1$ is isomorphic to
		\[
		\OO_{\gamma}(a_1 - a_2) \oplus \OO_{\gamma}(a_2)
		\] 
		under the identification of $\gamma_1$ with $\gamma$ under $r_1$.
		
		Repeat the process $m_1$-times and apply Lemma \ref{lem:blpNormalBundle} to conclude that the normal bundle of the minimal section $\gamma_{m_i}$ in this step is
		\[
		\OO_{\gamma}(a_1 - m_1a_2) \oplus \OO_{\gamma}(a_2) = \OO_{\gamma}(a_2) \oplus \OO_{\gamma}(a_3).
		\] 
		Repeating in total of $k = (m_1 + m_2 + m_3 + \ldots + m_n)$-times we get the the normal bundle of $\gamma_k$ in $U_k$ is
		\[
		\OO_{\gamma}(a_n) \oplus \OO_{\gamma}(m_na_{n+1}),
		\]
		which is balanced proving our claim.
	\end{proof}
	
	\begin{proposition}\label{prop:SQMci}
		Let $\gamma \subset X$ be a smooth curve on a threefold $X$ that is smooth along $\gamma$.
		Assume that there exists an open neighbourhood $\gamma \subset U \subset X$ and surfaces $S_1, S_2 \subset X$ so that $\gamma$ is the complete intersection of the $S_i \cap U$.
		
		Assume further that $\OO_{\gamma}(S_i)$ are proportional in $\Pic(\gamma)$, and that there exists a morphism $p\colon X \to Z$ contracting $\gamma$ to a point.
		Then there exists an SQM 
		\[
		\xymatrix@R=.4cm@C=0cm{
			\gamma \subset & X \ar@{..>}[rrrr]^{\chi} \ar[rrd]_p &&&& X^+ \ar[lld]^q &  \supset \gamma^+\\
			&&& Z
		}
		\]
		centered at $\gamma$.
		Moreover, the restriction of $\chi$ on $S_i$ contracts $\gamma$ and the strict transforms of the $S_i$ do not intersect along $\gamma^+$ (see Figure \ref{fig:SQM}).
	\end{proposition}
	
	\begin{proof}
		Since $\OO_{\gamma}(S_i)$ are proportional, there exists a divisor $D \in \Pic(\gamma)$ so that $\OO_{\gamma}(S_i) \isom \OO_{\gamma}(a_iD)$;
		without loss of generality assume that $a_i$ are positive with $a_1 \geq a_2$.
		Denote by $r\colon W \to X$ the blowup of $X$ along $\gamma$ with exceptional divisor $F \isom \PP\left(\OO_{\gamma}(a_1D)\oplus \OO_{\gamma}(a_2D)\right)$.
		Keeping the notation of Lemma \ref{lem:contracitonDecompBundle} we will denote by $\gamma_i \subset F$ the sections corresponding to the bundles $\OO_\gamma(S_i)$.	
		We will prove the statement by induction on $k\defeq a_1 - a_2$.
		Note that by Corollary \ref{cor:blowupUntilBalanced} we can always reduce to the case $k=0$, so that will be our base case.
		
		Suppose that $k = 0$, so that	
		Then $F \sim \gamma \times \PP^1$.
		Denote by $\pi\colon F \to \PP^1$ the projection to the second factor.
		Then, for a sufficiently ample divisor $A$ on $Z$, $D \defeq \tilde{S_2} + r^*p^*A$ is very ample away from $F$ and 
		\[
		D|_F = \pi^*\OO_{\PP^1}(1). 
		\]
		In particular $D$ is semiample and the associated contraction $W \to X^+$ gives is the desired morphism.
		Note that $F$ and $\tilde{S_i}$ intersect along the sections corresponding to $\OO_\gamma(S_i)$ and so the restriction of $W \to X^+$ on the $\tilde{S_i}$ contracts these sections.
		Furthermore these sections are disjoint proving the last assertion too.
		
		Now assume that the statement is true for all values less than $k$.
		By Lemma \ref{lem:contracitonDecompBundle} the restriction of $\tilde{S_2}$ on $F$ is a semiample, whose corresponding contraction contracts precisely $\gamma_1$.
		Therefore, for a sufficiently ample divisor $A_Z$ on $Z$, $D  \defeq \tilde{S_2} + r^*p^*A_Z$ is semiample, and very ample away from $F$.
		Denoting by $p_1\colon W \to \hat{Z}$ its corresponding contraction as well as applying Lemma \ref{lem:blpNormalBundle} with the induction hypotheses we get a diagram
		\[
		\xymatrix@R=.1cm@C=.1cm{
			\gamma_1 \subset & W\ar[dd]_r \ar@{..>}[rrrr]^{\hat{\chi}\,\,} \ar[rrd]_{p_1} &&&& W^+ \ar[lld]^{q_1} &  \supset \gamma_1^+\\
			&&& \hat{Z}\ar[dd]_{\eta}\\
			& X \ar[rrd]_{p\,\,\,}\\
			&&& Z
		}
		\]	
		
		Consider the divisor $\tilde{S_1}^+ \defeq \hat{\chi}_*\tilde{S_1}$.
		Since $\tilde{S_1}$ is negative against $\gamma_1$, $\tilde{S_1}^+$ is positive against $\gamma_1^+$.
		On the other hand, $\tilde{S_1}$ being trivial against $\gamma_2$ and $\gamma_2$ being disjoin from the indeterminacy locus of $\hat{\chi}$, $\tilde{S_1}^+$ is trivial against $\gamma_2^+$ without containing any of these curves in its base locus.
		Therefore, again adding a sufficiently ample divisor $A_Z$ from $Z$, we obtain $D^+ \defeq \tilde{S_1}^+ + q_1^*\eta^*A_Z$, which is semiample and zero against curves covering $\hat{\chi}(F)$.
		The contraction $W^+ \to X^+$ corresponding to $D^+$ completes the diagram into
		\[
		\xymatrix@R=.1cm@C=.1cm{
			& W\ar[dd]_r \ar@{..>}[rrrr]^{\hat{\chi}\,\,} \ar[rrd]_{p_1} &&&& W^+\ar[dd]^s \ar[lld]^{q_1} &\\
			&&& \hat{Z}\ar[dd]_{\eta}\\
			\gamma \subset& X \ar[rrd]_{p\,\,\,}&&&& X^+\ar[lld]^{\,q} & \supset \gamma^+\\
			&&& Z
		}
		\]	
		with the induced map $\chi\colon X \psmap X^+$ over $Z$ being the desired SQM.
		By the inductive hypothesis $\hat{\chi}$ restricted on $\tilde{S_1}$ contracts $\gamma$ and, by construction $s$ restricted on $\tilde{S_1}^+$ is an isomorphism.
		This shows that $\chi$ restricted on $S_1$ contracts $\gamma$.
		On the other hand, the restriction of $X \rmap W^+$ on $S_2$ is an isomorphism, while $s$ contracts $\gamma_2^+$ which proves the same assertion for $\chi$ restricted to $S_2$.
		Finally, once again by the inductive hypothesis, $\tilde{S_1}^+$ and $F^+$ do not intersect along $\gamma_1^+$;
		however $s\big(\tilde{S_2}^+\big)$ intersects $\gamma^+$ precisely at the image of $F^+$, which completes the proof. 
	\end{proof}

	\begin{figure}[h!]
		\begin{tikzpicture}
			\node[inner sep=.3cm] (X) at (-3,0)
			{\includegraphics[width=2cm]{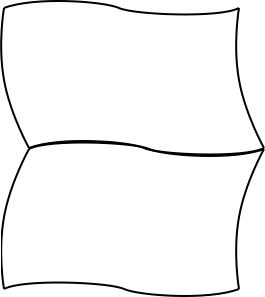}};
			
			\node[inner sep=.3cm] (S1) at (-3,.5) {$S_1$};
			\node[inner sep=.3cm] (S2) at (-3,-.5) {$S_2$};
			\node[inner sep=.3cm] (gamma) at (-4.2,0) {$\gamma$};
			
			\node[inner sep=.3cm] (W) at (-3,4)
			{\includegraphics[width=2cm]{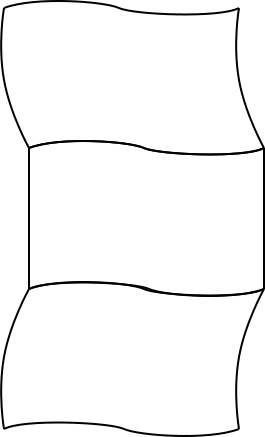}};
			
			\node[inner sep=.3cm] (S1) at (-3,5.1) {$\tilde{S_1}$};
			\node[inner sep=.3cm] (F) at (-3,4) {$F$};
			\node[inner sep=.3cm] (S2) at (-3,2.9) {$\tilde{S_2}$};
			\node[inner sep=.3cm] (gamma) at (-4.3,4.5) {$\gamma_1$};
			\node[inner sep=.3cm] (gamma) at (-4.3,3.45) {$\gamma_2$};
			
			\node[inner sep=.3cm] (W+) at (3,4)
			{\includegraphics[width=2cm]{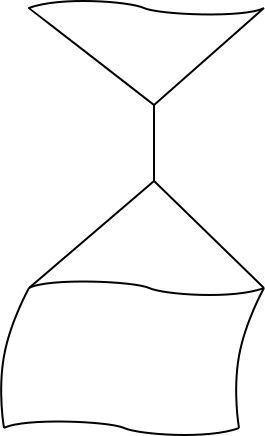}};
			
			\node[inner sep=.3cm] (S1) at (3.2,5.25) {$\tilde{S_1}^{\hspace{-1.2mm}+}$};
			\node[inner sep=.3cm] (F+) at (3.2,3.8) {$F^+$};
			\node[inner sep=.3cm] (S2) at (3,2.9) {$\tilde{S_2}^{\hspace{-1.2mm}+}$};
			\node[inner sep=.3cm] (gamma_1) at (3.6,4.5) {$\gamma_1^+$};
			\node[inner sep=.3cm] (gamma_2) at (4.3,3.45) {$\gamma_2^+$};
			
			\node[inner sep=.3cm] (X+) at (3,0)
			{\includegraphics[width=1.8cm]{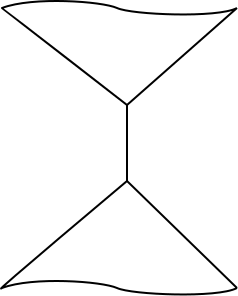}};
			
			\node[inner sep=.3cm] (S1) at (3.1,.75) {$S_1^+$};
			\node[inner sep=.3cm] (S2) at (3.07,-.75) {$S_2^+$};
			\node[inner sep=.3cm] (gamma) at (3.5,0) {$\gamma^+$};
			
			\draw[->] (W)--(X);
			\draw[->] (W+)--(X+);
			\draw[->,dotted] (W)--(W+) node[midway,above] {$\hat{\chi}$};
			\draw[->,dotted] (X)--(X+) node[midway,above] {$\chi$};
		\end{tikzpicture}
		\caption{The inductive step in the proof of Proposition \ref{prop:SQMci}}
		\label{fig:SQM}
	\end{figure}
	
	\begin{lemma}\label{lem:proportionality}	
		In the setting of Proposition \ref{prop:SQMci}, if $X$ is has Picard rank $2$ and $\OO_X(S_i)$ are $\QQ$-Cartier then the assumption on $\OO_{\gamma}(S_i)$ is superfluous;
		that is $\OO_{\gamma}(S_i)$ are proportional.
	\end{lemma}
	
	\begin{proof}
		If the $\OO_X(S_i)$ are proportional in $\Pic(X)$ then clearly their restrictions are proportional in $\Pic(\gamma)$.
		If not let $D_Z$ be any Cartier divisor on $Z$ and $D \defeq p^*D_Z$.
		Since $\rho(X) = 2$ and $\OO_X(S_i)$ are $\QQ$-Cartier and not proportional, they span $\Pic(X)$;
		in particular there exists $a_1, a_2$ so that $D \sim \OO_X(a_1S_1) \otimes \OO_X(a_2S_2)$.
		Restricting on $\gamma$ we obtain
		\[
		\OO_{\gamma} \sim \OO_{\gamma}(D) \sim \OO_{\gamma}(a_1S_1) \otimes \OO_{\gamma}(a_2S_2),
		\]
		i.e.\ $\OO_{\gamma}(S_1)$ is proportional to $\OO_{\gamma}(S_2)$ in $\Pic(\gamma)$.
	\end{proof}

	\subsection{Potential contractibility}
	
	\begin{definition}
		Let $\gamma \subset X$ be a curve in a $\QQ$-factorial projective variety $X$.	
		We say that $\gamma$ is \emph{potentially contractible} if for every SQM $f_i\colon X \psmap X_i$ so that ${f_i}_*\gamma$ generates an extremal ray $R^i$ of $\NEb(X_i)$, $R^i$ is contractible.	
	\end{definition}
	
	\begin{remark}\label{rem:everyCurvePC}
		If $X$ is a Mori Dream Space then \emph{every} curve $\gamma \subset X$ is potentially contractible:
		if there are no SQMs from $X$ so that the strict transform of $\gamma$ generates an extremal ray, then we are automatically done;
		on the other hand if $f_i\colon X \psmap X_i$ is such an SQM then, by Remark \ref{rem:MDS}\eqref{it:MDSrem2}, $f_i$ is one of the SQMs of Definition \ref{def:MDS}\eqref{it:MDS3};
		in particular, by Definition \ref{def:MDS}\eqref{it:MDS2}, any nef divisor on $X_i$ that is perpendicular to the ray $R^i$ generated by ${f_i}_*\gamma$ is semiample, i.e.\ $R^i$ is contractible.
	\end{remark}
	
	\begin{proposition}\label{prop:potContr=>MDS}
		Let $\Gamma$ be a configuration of smooth pairwise skew curves $\gamma_i$ and $C \link{S_1,S_2} \Gamma$  be a skew $C$-rigid linkage.
		Denote by $X \to \PP^3$ the blowup along $C$.
		
		Then $X$ is a Mori Dream Space if and only if the curves $\gamma_i \subset X$ are potentially contractible.
	\end{proposition}

	\begin{proof}
		Assuming $X$ is a Mori Dream Space we conclude by Remark \ref{rem:everyCurvePC}.
		
		Suppose that $\gamma_i \subset X$ are potentially contractible.
		We will be working under the assumptions of Setup \ref{set:rigidLinkage}.
		Denote by $R_a$ the ray in $\NEb(X)$ generated by curves $\gamma_i$ whose indices lie $I_a$, $1 \leq a \leq k$ (see Lemma \ref{lem:skewLinkageInters} and Figure \ref{fig:NE(X)}).	
		We consider the rational map $X \rmap W$ given by the sections of $S_2$.
		We will show that $X \rmap W$ is a rational contraction factoring as
		\[
		\xymatrix@R=.2cm@C=.2cm{
			&X \ar[rd] \ar[ld] \ar@{..>}[rr]^{\chi_1\,\,} && X_1 \ar[ld] \ar@{..>}[rr]^{\,\,\,\,\chi_2} \ar[rd] && \ar[ld]&\dots& \ar[rd] \ar@{..>}[rr]^{\chi_{k}\quad}&& X_{k} \ar[ld] \ar[rd]
			\\
			\PP^3&& Z_1 && Z_2 &&&& Z_{k} && W,
		}
		\]
		where $\chi_a$ is the flip of the ray $R_a$ via the construction of Proposition \ref{prop:SQMci}.
		
		Recall that, for all $i \in I$, we may define the open subsets
		\[
		U_i \defeq X \setminus \underset{j\neq i}{\bigcup} \gamma_i
		\]
		so that $\gamma_i \subset U_i$ and $\gamma_i$ is the complete intersection of the surfaces $S_i \cap U_i$.
		
		We begin by flipping the ray $R_1$.
		By Lemma \ref{lem:skewLinkageInters} and Proposition \ref{prop:conesRigidLinkage} $R_1$ is extremal and thus by assumption contractible;
		let $X \to Z_1$ be the contraction.
		Since $X$ has Picard rank $2$ and is $\QQ$-factorial, by Lemma \ref{lem:proportionality}, all assumptions of Proposition \ref{prop:SQMci} are satisfied, giving us desired SQM $\chi_1\colon X \psmap X_1$.
		
		If $\OO_{X_1}(S_2)$ is semiample on $X_1$, then we are done.
		If not, let $\gamma$ be a curve on the stable base locus of $\OO_{X_1}(S_2)$ spanning an extremal ray $R_2$ of $\NE(X_1)$.
		Then $\gamma$ is contained in both $S_1$ and $S_2$.
		However, by the construction of Proposition \ref{prop:SQMci}, $\gamma$  has to be one of the components $\gamma_i$ of $\Gamma$.
		It is thus by assumption contractible and, once more by Lemma \ref{lem:proportionality} and Proposition \ref{prop:SQMci}, we may flip $R_2$.
		
		Applying the same argument iteratively we conclude.
	\end{proof}
	
	\begin{remark}\label{rem:naming}
		The type of the contraction $X_k \to W$ depends on the type of the linkage.
		If $C \link{S_1,S_2} \Gamma$ is unbalanced, then $X_k \to W$ is a divisorial contraction with exceptional divisor $S_1$; 
		if it is furthermore super-rigid, then $S_1$ is contracted to a point.
		On the other hand, if $C \link{S_1,S_2} \Gamma$ is balanced, $X_k \to W$ is of fiber type; if it is super-rigid then $W = \PP^1$.
	\end{remark}
	
	We now prove a partial inverse to Proposition \ref{prop:semiAmp=>Qcan}.
	
	\begin{proposition}\label{prop:Qcan=>potContr}
		Let $\Gamma$ be a configuration of smooth pairwise skew curves $\gamma_i$
		and \mbox{$C \link{S_1,S_2} \Gamma$} be a skew $C$-rigid linkage.
		Denote by $X \to \PP^3$ the blowup along $C$.	
		Then, for any $1 \leq i \leq m$, if $\gamma_i$ is potentially contractible, it is $\QQ$-canonical.
		
		Conversely, if for all $i \in I_1 \sqcup \ldots \sqcup I_{k-1}$ we have that
		\begin{equation}\label{eq:Qcan=>potContr}\tag{$\star$}
			\gamma_i \text{ is $\QQ$-canonical}\quad \text{ and } \quad 4(g_i-1) - (n_2-4)d_i < 0
		\end{equation}
		and either
		\begin{enumerate}
			\item\label{it:Qcan=>potContr1} $C \link{S_1,S_2} \Gamma$ is super-rigid and $\gamma_i$ with $i\in I_k$ also satisfy \eqref{eq:Qcan=>potContr} or 
			\item\label{it:Qcan=>potContr2} for all $i \in I_k$, $\gamma_i$ is $(n_2-4)$-subcanonical, i.e.\ $K_{\gamma_i} \sim (n_2-4)H|_{\gamma_i}$,
		\end{enumerate}
		then the $\gamma_i$ are potentially contractible.
	\end{proposition}

	\begin{proof}
		The proof that potential contractibility implies $\QQ$-canonicity is similar to that of Proposition \ref{prop:semiAmp=>Qcan}.
		
		The reverse implication is a matter of verifying the assumptions of Propositions \ref{prop:contractionNeg} and \ref{prop:contractionZero}.
		First note that for all $j \in  I_1 \sqcup \ldots \sqcup I_{k-1}$ and $i \in I_k$, by Lemma \ref{lem:skewLinkageInters}, we have $f \prec \gamma_j \prec \gamma_i$, therefore $S_i\cdot \gamma_j < S_i \cdot \gamma_i \leq e_i \leq 0$.
		For any $j \in I_1$ the curve $\gamma_j$ spans an extremal ray of $X$ and	we may calculate that
		\[
		S_1\cdot \gamma_j = \big(S_2 - (n_2 - n_1)H\big)\cdot \gamma_j = e_j - (n_2 - n_1)d_j = 2g_j - 2 - (n_2 - 4)d_j
		\]
		and consequently
		\[
		(2g_j - 2) + S_1\cdot \gamma_j =  4(g_j - 1) - (n_2 - 4)d_j < 0.
		\]
		Finally, similarly to proof of Proposition \ref{prop:semiAmp=>Qcan}, we may calculate that $\QQ$-canonicity of $\gamma_j$ is equivalent to $\OO_{\gamma_j}(D_1) \isom \OO_{\gamma_j}$, which in turn implies that $h^0(\Gamma,D_1) \neq 0$.
		Then, by Proposition \ref{prop:contractionNeg}, $D_1$ is semiample and $\gamma_i$ is contractible.
		For $j \in I_2 \sqcup \ldots \sqcup I_{k-1}$ the argument is the same: we first use Proposition \ref{prop:SQMci} to flip the previous rays and then verify the condition of Proposition \ref{prop:contractionNeg}.
		If we are furthermore in case \eqref{it:Qcan=>potContr1} then the same is true if $j \in I_k$.
		
		Suppose now that we are in case \eqref{it:Qcan=>potContr2} and $i \in I_k$.
		By restricting $\OO_X(S_2)$ first on $S_1$ and then on $\gamma_i$ we obtain that $\OO_{\gamma_i}(S_2) = \OO_{\gamma_i}(\gamma_i)$.
		By the adjunction formula on $S_1$ we have
		\[
		\OO_{\gamma_i}(\gamma_i) = K_{\gamma_i} - K_{S_1}|_{\gamma_i} = K_{\gamma_i} - (n_1-4)H|_{\gamma_i}
		\]
		which by our assumption coincides with $\OO_{\gamma_i}$.
		Therefore $h^0(\Gamma,S_2) \neq 0$.
		Moreover 
		\[
		h^1(X,S_2-S_1) = h^1(X,H) = 0 \quad \text{ and } \quad h^1(S_1,\OO_{S_1})=0,
		\]
		the latter since $S_1$ the strict transform of a smooth hypersurface in $\PP^3$.
		By Proposition \ref{prop:contractionZero}, we conclude that $D_k =  S_2$ is base point free, i.e.\ $\gamma_i$ is contractible. 
	\end{proof}
	
	\subsection{Degeneration phenomena and some interesting examples}
	
	Proposition \ref{prop:Qcan=>potContr} allows us to exhibit examples of nef and big, non-semiample divisors degenerating to a semi-ample ones.
	
	\begin{corollary}\label{cor:degenToSemiample}
		Let $g',d', n_1, n_2$, $g,d$ and $U_{g,d}$ be as in Theorem \ref{thm:mainThmRigid}.
		
		Then, for sufficiently large $n_1, n_2$, there exists a subset $V_{g,d}$ of $\overline{U}_{g,d}$, of positive codimension, so that for any $C \in V_{g,d}$, the blowup $X \to \PP^3$ along $C$ \emph{is} a Mori Dream Space;
		in particular, every nef divisor on $X$ is semiample.
	\end{corollary}
	
	\begin{proof}
		Define
		\[
		V_{g',d'} = \left\{ C' \in H_{g',d'}^S \,|\, C' \text{ is $\QQ$-canonical }\right\}
		\]
		and take $V_{g,d}$ to be the $(n_1,n_2)$-linked family (see subsection \ref{subsec:linkageAndHilberFlag}).
		Then, for sufficiently large $n_1,n_2$, the conditions of Propositions \ref{prop:specialization} and \ref{prop:Qcan=>potContr} are satisfied.
		On the one hand this implies that any curve in $V_{g,d}$ is the specialization of a curve in $U_{g,d}$, showing that $V_{g,d} \subset \overline{U}_{g,d}$;
		on the other hand for any $C \in V_{g,d}$, the blowup along $C$ is a Mori Dream Space.
	\end{proof}
	
	Finally, by appropriately choosing skew $\QQ$-canonical curves $\gamma_1, \dots, \gamma_k$ and considering the blowup along the linked curve, we may exhibit examples of: 
	\begin{enumerate}
		\item\label{it:interestingEx1} threefolds of Picard rank $2$, whose movable cone has arbitrarily many Mori chambers;
		\item\label{it:interestingEx2} SQMs that flip arbitrarily many curves;
		\item\label{it:interestingEx3} SQMs that flip curves with arbitrarily unbalanced normal bundles, that is $N_{\gamma/X} = \OO_{\gamma}(S_1) \oplus \OO_{\gamma}(S_2)$ with $\deg \OO_{\gamma}(S_2-S_1)$ arbitrarily large.
	\end{enumerate}
	
	Indeed, for \eqref{it:interestingEx1}, choosing the curves $\gamma_i$ so that for all the ratios $\frac{e_i}{d_i}$ of Setup \ref{set:rigidLinkage} are different, implies that all none of the $\gamma_i$ are numerically proportional; 
	in that case $\Mov(X)$ has exactly $k$ Mori chambers (see Figure \ref{fig:Eff(X)}).
	On the other hand, if  the curves are chosen so that the ratios $\frac{e_i}{d_i}$ are all equal, then all the $\gamma_i$ lie on the same ray $R$;
	the SQM flipping $R$, flips $k$ distinct curves, showing \eqref{it:interestingEx2}.
	For example, we may choose the curves $\gamma_i$ to be rational of degrees $d_i$, with $d_i$ all different for \eqref{it:interestingEx1} and all the same for \eqref{it:interestingEx2}.
	
	If $n_1, n_2$ are the degrees of the corresponding linkage $C \link{n_1,n_2} \Gamma$, then for every flipped curve we have $N_{\gamma_i/X} = \OO_{\gamma}(S_1) \oplus \OO_{\gamma}(S_2)$;
	therefore $\deg(\OO_{\gamma}(S_2-S_1)) = (n_2-n_1)d_i$.
	Choosing the difference $n_2 - n_1$ to be arbitrarily large we conclude \eqref{it:interestingEx3}.
	Note that by also including some non $\QQ$-canonical curves to $\Gamma$ we may exhibit these examples on varieties that are not Mori Dream Spaces.

	\bibliography{nonMDSbib}{}
	\bibliographystyle{alpha}   
	
\end{document}